\documentclass[11pt]{amsart}

\usepackage[a4paper,vmargin=4cm]{geometry}
\usepackage{amsfonts,amssymb,amscd,amstext}
\usepackage{graphicx}
\usepackage{color}
\usepackage[dvips]{epsfig}

\usepackage[utf8]{inputenc}
\usepackage{hyperref}

\usepackage{verbatim}
\usepackage{fancyhdr}
\input xy
\xyoption{all}

\usepackage{times}
\usepackage{enumerate}
\usepackage{titlesec}
\usepackage{mathrsfs}

\numberwithin{equation}{section}
\numberwithin{figure}{section}

\titleformat{\section}
{\filcenter\bfseries\large} {\thesection{.}}{0.2cm}{}
\titleformat{\subsection}[runin]
{\bfseries} {\thesubsection{.}}{0.15cm}{}[.]
\titleformat{\subsubsection}[runin]
{\em}{\thesubsubsection{.}}{0.15cm}{}[.]

\usepackage[up,bf]{caption}

\newtheorem{theorem}{Theorem}[section]
\newtheorem{proposition}[theorem]{Proposition}

\newtheorem{lemma}[theorem]{Lemma}
\newtheorem{corollary}[theorem]{Corollary}

\theoremstyle{definition}
\newtheorem{definition}[theorem]{Definition}
\newtheorem{remark}[theorem]{Remark}

\newtheorem{problem}[theorem]{Problem}
\newtheorem{example}[theorem]{Example}

\newcommand\Kcal{\mathcal{K}}

\newcommand\Pcal{\mathcal{P}}

\newcommand\bx{\mathbf{x}}

\newcommand\Ascr{\mathscr{A}}

\newcommand\Cscr{\mathscr{C}}

\newcommand\Oscr{\mathscr{O}}

\newcommand\C{\mathbb{C}}
\newcommand\D{\overline{\mathbb D}}
\newcommand\CP{\mathbb{CP}}
\renewcommand\D{\mathbb D}

\renewcommand\H{\mathbb{H}}

\newcommand\N{\mathbb{N}}

\newcommand\R{\mathbb{R}}

\newcommand\U{\mathbb{U}}
\newcommand\Z{\mathbb{Z}}

\newcommand\ggot{\mathfrak{g}}

\newcommand\igot{\mathfrak{i}}

\renewcommand\igot{\mathfrak{i}}

\newcommand\Igot{\mathfrak{I}}

\newcommand\E{\mathrm{e}}

\renewcommand\imath{\igot}

\newcommand\hra{\hookrightarrow}

\newcommand\wt{\widetilde}

\newcommand\di{\partial}
\newcommand\dibar{\overline\partial}

\newcommand\nullq{{\mathbf A}}

\newcommand\GCMI{\mathrm{GCMI}}

\newcommand\Aut{\mathrm{Aut}}

\def\Ell1{\mathrm{Ell_1}}
\def\CEll1{\mathrm{CEll_1}}

\begin{document}

\fancyhead[LO]{Minimal surfaces with symmetries}
\fancyhead[RE]{F.\ Forstneri\v c}
\fancyhead[RO,LE]{\thepage}

\thispagestyle{empty}


\begin{center}
{\bf \LARGE Minimal surfaces with symmetries}

\vspace*{0.5cm}

{\large\bf  Franc Forstneri\v c}  
\end{center}

\vspace*{0.5cm}

{\small
\noindent {\bf Abstract} 
Let $G$ be a finite group acting on a connected open Riemann surface $X$ by 
holomorphic automorphisms and acting on a Euclidean space $\R^n$ $(n\ge 3)$ 
by orthogonal transformations. We identify a necessary and sufficient condition for the 
existence of a $G$-equivariant conformal minimal immersion $F:X\to\R^n$.
We show in particular that such a map $F$ always exists if $G$ acts without 
fixed points on $X$. Furthermore, every finite group $G$ arises in this way for 
some open Riemann surface and $n=2|G|$. We obtain an analogous result
for minimal surfaces having complete ends with finite total Gaussian curvature,
and for discrete groups acting on $X$ properly discontinuously 
and acting on $\mathbb R^n$ by rigid transformations. 
}

\noindent{\bf Keywords:} \hspace*{0.1cm} 
minimal surface, equivariant conformal minimal immersion, Oka manifold

\vspace*{0.1cm}

\noindent{\bf MSC (2020):}\hspace*{0.1cm}  Primary 53A10; Secondary 32E30, 32H02, 32Q56
%
%
%
%
%

\noindent {\bf Date: \rm 16 February 2024}


%
%

\section{Introduction}\label{sec:intro}  
Objects with symmetries are of special interest in any mathematical theory.
In this paper, we study the existence of immersed orientable  
minimal surfaces in Euclidean spaces $\R^n$ with a given finite or countable 
group of symmetries induced by rigid transformations of $\R^n$.

An immersed minimal surface in $\R^n$ for $n\ge 3$ is the image of a conformal 
harmonic immersion $F:X\to\R^n$ from an open conformal surface $X$,
which can be taken to be a Riemann surface if it is orientable; see 
\cite{Osserman1986,AlarconForstnericLopez2021}.
We shall call such $F$ a {\em conformal minimal immersion}.
Euclidean isometries of $\R^n$ form an affine group generated by the orthogonal 
group $O(n,\R)$ and the additive group $(\R^n,+)$ acting by translations. 
By also adding dilations, we obtain the group of {\em rigid transformations}.
Postcomposition by a rigid transformation of $\R^n$ 
maps minimal surfaces to minimal surfaces,
and rigid transformations are the largest class of self-maps of $\R^n$ with this property. 
Hence, it is of interest to find minimal surfaces which are invariant under a given group 
of rigid transformations. 
Symmetries of specific minimal surfaces were studied by many authors, but 
we are interested in general existence results. 

Let $X$ be a connected open Riemann surface and $G$ be a finite subgroup of 
the group $\Aut(X)$ of holomorphic automorphisms of $X$.
The stabiliser $G_x=\{g\in G:gx=x\}$ of any point $x\in X$ is a cyclic subgroup of $G$,   
which is trivial for points in the complement of a closed
discrete subset of $X$ (see \cite[Corollary 3.5, p.\ 93]{Miranda1995}). 
Assume that $G$ also acts on $\R^n$ by orthogonal transformations.
The following result and Remark \ref{rem:necessary} provide 
a necessary and sufficient condition for the existence of a $G$-equivariant 
conformal minimal immersion $X\to \R^n$.

%
%
\begin{theorem}\label{th:main}
Let $G$ be a finite group acting effectively on a connected open 
Riemann surface $X$ by holomorphic automorphisms, and acting on 
$\R^n$ $(n\ge 3)$ by orthogonal transformations. 
If for every nontrivial stabiliser $G_x$ $(x\in X)$ there is a 
2-plane $\Lambda_x\subset \R^n$ on which $G_x$ acts effectively by rotations, 
then there exists a conformal minimal immersion $F:X\to \R^n$ such that 
\begin{equation}\label{eq:equivariant}
	F(gx) = gF(x)\ \ \text{holds for all $x\in X$ and $g\in G$},
\end{equation}
and the image $F(X)$ is not contained in any affine hyperplane of $\R^n$.
\end{theorem}

A map $F:X\to \R^n$ is said to be {\em $G$-equivariant} if 
condition \eqref{eq:equivariant} holds, and {\em nondegenerate} if $F(X)$ 
is not contained in any affine hyperplane of $\R^n$.
Note that the image $S=F(X)\subset \R^n$ of a $G$-equivariant map 
is $G$-invariant, i.e., $gS=S$ for all $g\in G$.
Since the fixed-point-set of a nontrivial linear map on $\R^n$ is
a proper linear subspace of $\R^n$, nondegeneracy of $F$ implies that every
$g\in G$ which acts effectively on $\R^n$ also acts effectively on 
$F(X)\subset \R^n$. Thus, if $H$ is the normal subgroup of $G$
consisting of all elements $g\in G$ which act trivially on $\R^n$, then
$G/H$ is a symmetry group of the minimal surface $F(X)$ in Theorem \ref{th:main}.
Any additional symmetries can be eliminated by using a general position
argument in the proof of the theorem. 

Our proof of Theorem \ref{th:main} gives several additions concerning approximation, 
interpolation, and the control of the flux; see Theorem \ref{th:mainbis} and compare 
with the results in the non-equivariant case  
\cite[Theorems 3.6.1 and 3.6.2]{AlarconForstnericLopez2021}.
In particular, the map $F$ in Theorem \ref{th:main} can be chosen to be 
the real part of a $G$-equivariant null holomorphic immersion $X\to\C^n$;
see Theorem \ref{th:null}. 

In Section \ref{sec:FTC} we construct $G$-equivariant minimal surfaces with 
complete ends of finite total curvature on given finitely many orbits of $G$
on the Riemann surface $X$; 
see Theorem \ref{th:FTC} and Corollary \ref{cor:FTC}.
However, we do not know whether all ends can be made complete
and of finite total curvature; see Problem \ref{prob:FTC}. 
In particular, the construction of complete $G$-equivariant minimal 
surfaces of finite total curvature remains an open problem.

Finally, in Section \ref{sec:infinite} we show that the analogue of Theorem \ref{th:main} 
also holds if $G$ is an infinite discrete group acting on $\R^n$ by rigid transformations, 
and acting on a Riemann surface $X$ properly discontinuously by holomorphic 
automorphisms such that the quotient surface $X/G$ is noncompact; 
see Theorem \ref{th:infinite}. This case is only relevant if $X$ has genus at most one, 
since every Riemann surface of genus $\ge 2$ has at most finitely many 
automorphisms by a theorem of Hurwitz \cite{Hurwitz1893}
(see also \cite[Theorem 3.9]{Miranda1995}). 

%
%
Minimal surfaces with symmetries appeared in the very origin of the theory; 
indeed, most classical examples have symmetries (the catenoid, the helicoid, 
Scherk's surfaces, Riemann's minimal examples, Schwarz's surfaces, etc.).
All mentioned examples have infinite groups of symmetries and are parameterized
by plane domains. See also the discussion in Example \ref{ex:sphere}. 
Finding examples with given groups of symmetries from  
Riemann surfaces of genus $\ge 1$ is a more difficult task due to
the problem of controlling the periods of their Weierstrass data, 
and only a few examples have been described explicitly. 
In this paper we give general existence results for such minimal surfaces. 
The techniques developed in the paper also 
seem promising for constructing minimal surfaces with given 
symmetries and satisfying various additional conditions
such as being complete or proper;
see Problems \ref{prob:CY} and \ref{prob:proper}.

%
%
\begin{remark}\label{rem:necessary}
The conditions on stabilisers  in Theorem \ref{th:main} are necessary. Indeed, 
let $x\in X$ be a point with a nontrivial stabiliser $G_{x}$ of order $k>1$. 
There is a local holomorphic coordinate $z$ on a neighbourhood $U\subset X$ 
of $x$, with $z(x)=0$, in which a generator $g$ of the cyclic group 
$G_x$ is the rotation $g z = \E^{\imath \phi}z$ through  
the angle $\phi=2\pi/k$ (see Miranda \cite[Corollary 3.5, p.\ 93]{Miranda1995}). 
Assume that $F:X\to \R^n$ is a $G$-equivariant conformal immersion,  
not necessarily harmonic. Differentiating the identity \eqref{eq:equivariant} and
taking into account that $g$ acts linearly on $\R^n$ gives  
\[
	dF_{x} \circ dg_{x} = g\circ dF_{x}:
	T_{x}X \to \Lambda_x:= dF_{x}(T_{x}X)\subset \R^n.
\]
Since $dF_{x} : T_{x}X \to \Lambda_x$ is a conformal linear isomorphism, 
we infer that $\Lambda_x$ is a $G_x$-invariant $2$-plane in $\R^n$ on which 
$g$ acts as the rotation $R_\phi$ through the angle $\phi$, 
so the conditions in Theorem \ref{th:main} hold. 
Conversely, these conditions imply that the local conformal linear
embedding $U\to \Lambda$ is $G_x$ equivariant.  
These conditions are superfluous for minimal surfaces with branch points. 
\end{remark}

%
%
\begin{remark}\label{rem:disconnectedX}
In Theorem \ref{th:main} and its corollaries presented below, 
the Riemann surface $X$ is assumed to be connected. However, 
these results generalize to the case when for every connected component $X'$ of $X$
the stabiliser group $G_{X'}=\{g\in G: gX'=X'\}$ acts effectively on $X'$. 
This is equivalent to asking that the stabiliser $G_x$ of a generic point $x\in X$ is trivial.

The first immediate reduction is to the case when
$G$ acts transitively on the set of connected components of $X$.
(See the argument preceding 
\cite[Theorem 4.1]{KutzschebauchLarussonSchwarz2021JGEA}
by Kutzschebauch et al.)
Assuming this to be the case, fix a component $X'$ of $X$.
Since $G_x\subset G_{X'}$ holds for every $x\in X'$,
Theorem \ref{th:main} provides a $G_{X'}$-equivariant conformal
minimal immersion $F':X'\to\R^n$. If $X''$ is another component of $X$
and $h\in G$ is such that $h(X'')=X'$, we define $F$ on $X''$ by
$F(x)=h^{-1}F'(hx)$ for $x\in X''$. It is immediate that the resulting map $F:X\to\R^n$
is a $G$-equivariant conformal minimal immersion.

If on the other hand the stabiliser $G_{X'}$ of some component $X'$ of $X$ does 
not act effectively on $X'$, then the conditions in Theorem \ref{th:main} must be adjusted.
We shall not consider this case.
\end{remark}

%
%
Theorem \ref{th:main} is a special case of Theorem \ref{th:mainbis}, 
which also involves approximation and interpolation 
of a given $G$-equivariant minimal immersion on a suitable $G$-invariant subset of $X$
by global $G$-equivariant conformal minimal immersions. 
The proof relies on two main ingredients.
The classical Enneper--Weierstrass representation of minimal surfaces
in $\R^n$ reduces the problem to constructing holomorphic maps
(the so-called Weierstrass data) from the given 
Riemann surface $X$ into the punctured null quadric $\nullq_*$ in $\C^n$ 
(see \eqref{eq:nullq}) having suitable integrals (periods) on a system
of curves in $X$ whose union contains a basis of the homology group $H_1(X,\Z)$ 
and some other arcs which are used to guarantee the interpolation conditions.
In our case, the Weierstrass data are $G$-equivariant holomorphic maps from $X$ 
to the projective compactification of the null quadric, and we formulate
a Weierstrass representation theorem for $G$-equivariant minimal surfaces;
see Theorem \ref{th:GWR}. The main point is to approximate
such maps on certain $G$-invariant Runge subsets of $X$ by global $G$-equivariant
holomorphic maps having suitable periods. We combine the approach 
developed in \cite{AlarconForstneric2014IM} 
(see also \cite[Theorem 3.6.1]{AlarconForstnericLopez2021}) 
with \cite[Theorem 4.1]{KutzschebauchLarussonSchwarz2021JGEA} 
due to Kutzschebauch et al., which shows how to 
reduce Oka-theoretic problems for certain $G$-equivariant holomorphic 
maps to the nonequivariant case for sections of an associated holomorphic map
having ramification points; see Section \ref{sec:towards}. 
Ultimately, the main complex-analytic tool that we use is an Oka-theoretic result for sections 
of ramified holomorphic maps (see \cite[Theorem 2.1]{Forstneric2003FM} and 
\cite[Theorem 6.14.6]{Forstneric2017E}), combined with the techniques from 
\cite{AlarconForstneric2014IM} and  \cite[Chapter 3]{AlarconForstnericLopez2021}
which enable us to control periods of maps $X\to\nullq_*$. The main step 
is Lemma \ref{lem:main}, and Theorem \ref{th:main} is then proved 
in Section \ref{sec:proof}. 
The same method applies if $G$ is an infinite discrete group acting on $X$
properly discontinuously; see Theorem \ref{th:infinite}.

In the remainder of this introduction, we give several corollaries to Theorem \ref{th:main}
and we place our results in the context of what is known.
The following corollary is immediate.  

\begin{corollary}\label{cor:free}
If $G$ is a finite group acting freely (without fixed points) on an open 
Riemann surface $X$ by holomorphic automorphisms, then for every action of 
$G$ by orthogonal maps on $\R^n$ $(n\ge 3)$ there exists a 
nondegenerate $G$-equivariant 
conformal minimal immersion $X\to \R^n$, which can be chosen to be the
real part of a $G$-equivariant null holomorphic immersion $X\to\C^n$.
\end{corollary}

If $G$ is a finite group acting on a connected Riemann surface $X$ by holomorphic 
automorphisms, then the union $X_0$ of fixed point sets of elements of $G$ is a closed
discrete subset of $X$ (see  \eqref{eq:X0}), which is finite if $X$ has finite genus but
may be infinite otherwise; see Section \ref{sec:prelim}. Removing from $X$ any closed 
$G$-invariant subset $X'$ containing $X_0$, the group $G$ acts freely on the open 
Riemann surface $X\setminus X'$, and hence Corollary \ref{cor:free} applies to 
the pair $(X\setminus X',G)$.

%
%
\begin{corollary}\label{cor:everyG}
For every connected open Riemann surface $X$ and finite subgroup $G\subset \Aut(X)$ 
of order $n\ge 2$ there are an effective action of $G$ by orthogonal transformations on 
$\R^{2n}$ and a nondegenerate 
$G$-equivariant conformal minimal immersion $X\to\R^{2n}$.
\end{corollary}

\begin{proof}
I wish to thank Urban Jezernik for the following argument. 
Consider the regular representation of $G$ on the complex Euclidean space $\C^n$
with the basis vectors $e_g$ for $g\in G$, 
where an element $h\in G$ acts by $h e_g=e_{hg}$.
For a fixed $g\in G$ of order $k>1$ let $\Sigma_g$ denote the $k$-dimensional
$\C$-linear subspace of $\C^n$ spanned by the vectors $e_{g^j}$ $(j=0,1,\ldots,k-1)$ 
corresponding to the elements of the cyclic group $\langle g\rangle$. 
Clearly, $\Sigma_g$ is $g$-invariant 
and the eigenvalues of the $\C$-linear isomorphism 
$g:\Sigma_g\to\Sigma_g$ are precisely all the $k$-th roots of $1$. 
In particular, there is a vector $0\ne w\in \Sigma_g$ with $gw=\E^{\imath 2\pi/k}w$.
Identifying $\C^n$ with $\R^{2n}$, the $2$-plane $\Lambda_g\subset \R^{2n}$ 
determined by the complex line $\C w$ is $g$-invariant 
and $g$ acts on it as a rotation through the angle $2\pi/k$.
Since every stabiliser $G_x$ in Theorem \ref{th:main} is a cyclic subgroup of $G$,
the conditions of Theorem \ref{th:main} hold for this representation of $G$.
\end{proof}

Given a smooth surface $X$, an immersion $F:X\to \R^n$ induces on $X$ 
a unique structure of a conformal surface such that $F$ is a conformal immersion
(see \cite[Sect.\ 1.10]{AlarconForstnericLopez2021}).
This conformal structure is clearly invariant under postcomposition of
$F$ by rigid motions of $\R^n$. 
In particular, if $F$ is an embedding and the image surface 
$\wt X=F(X)\subset \R^n$ is $G$-invariant for a finite subgroup $G$ of $O(n,\R)$ 
(i.e., $g\wt X=\wt X$ holds for all $g\in G$), 
there is a unique action of $G$ on $X$ by conformal automorphisms
such that $F$ is a $G$-equivariant conformal embedding. 
If in addition the surface $X$ is oriented and every $g\in G$ preserves 
the orientation on $X$, then $G$ acts on $X$ by holomorphic automorphisms.
If this action is effective, Remark \ref{rem:necessary} shows that the conditions 
on stabilisers hold, so Theorem \ref{th:main} implies the following corollary. 

%
%
\begin{corollary}\label{cor:special}
Assume that $G$ is a finite subgroup of the orthogonal group $O(n,\R)$ for some
$n\ge 3$ and $X \subset \R^n$ is a smoothly embedded, connected, 
oriented, noncompact, $G$-invariant surface  
such that every $g\in G$ preserves the orientation on $X$, 
and $g$ induces the identity map on $X$ only if $g=1\in G$. 
Then, $X$ endowed with the complex structure induced by the embedding 
$X\hra \R^n$ and with the induced action of $G$ on $X$ by holomorphic 
automorphisms admits a nondegenerate 
$G$-equivariant conformal minimal immersion $F:X\to\R^n$.
\end{corollary}

\begin{remark}\label{rem:homotopy}
In the context of Corollary \ref{cor:special}, it is natural to ask whether there is
a regular homotopy of $G$-equivariant (conformal) 
immersions $F_t:X\to \R^n$ $(t\in [0,1])$
connecting the initial embedding $F_0:X\hra\R^n$ to a conformal minimal immersion
$F_1:X\to\R^n$. An inspection of our proof of Theorem \ref{th:main} 
shows that there is a homotopy of $G$-equivariant maps $f_t:X\to Y$ $(t\in [0,1])$ 
into the projective closure $Y$ \eqref{eq:Y} of the null quadric 
such that $f_0=2\di F_0/\theta$ (see \eqref{eq:theta0} for the definition
of the holomorphic 1-form $\theta$ on $X$), the map $f_1$ is holomorphic, 
and $f_1=2\di F_1/\theta$. However, we do not know whether the maps $f_t$ 
for $0<t<1$ can be chosen such that they integrate to immersions $F_t:X\to\R^n$
with $2\di F_t=f_t \theta$. 
\end{remark}

%
%
\begin{example}[{\bf Equivariant minimal surfaces of genus zero}] \label{ex:sphere}
Let $S$ be the unit sphere in $\R^3$. The induced Riemann surface structure
on $S$ is that of the Riemann sphere $\C\cup\{\infty\}=\CP^1$, 
the unique complex structure on $S$ up to biholomorphisms. 
The special orthogonal group $SO(3,\R)$ acts on $S$ by orientation 
preserving isometries, hence by holomorphic automorphisms, and it 
forms a real 3-dimensional subgroup of the holomorphic automorphism group 
\[
	\Aut(S)=\Big\{z\mapsto \frac{az+b}{cz+d}:\ \ a,b,c,d\in\C,\ ad-bc=1\Big\}.
\]
Finite subgroups of $SO(3,\R)$ are called {\em spherical von Dyck groups}.
Besides the cyclic and the dihedral groups, there are 
the symmetry groups of Platonic solids, the so-called {\em crystallographic groups}: 
the alternating group $A_4$ of order $12$ is the group of symmetries of the tetrahedron, 
the symmetric group $S_4$ of order $24$ is the group of symmetries of the cube 
and the octahedron, and the alternating group $A_5$ of order $60$ is the group 
of symmetries of the icosahedron and the dodecahedron. 
Corollary \ref{cor:special} shows that every spherical von Dyck group of order 
$m>1$ is a group of symmetries of a minimal surface in $\R^3$ parameterized 
by a complement of $m$ points in $\CP^1$. However, this case is already known. 
After the initial work of Goursat \cite{Goursat1887}, it was shown by Xu \cite{Xu1995}, 
using explicit functions in the Enneper--Weierstrass representation,
that any closed subgroup $G$ of $SO(3,\R)$, 
which is not isomorphic to $SO(2,\R)$ or $SO(3,\R)$, is the symmetry group of a 
complete immersed minimal surface in $\R^3$ of genus zero with finite total curvature 
and embedded ends. In the genus zero case, the only period vanishing
conditions are those coming from the ends, which amount to 
vanishing of the residues of the Weierstrass data at such points. 
Examples of (families of) minimal surfaces in $\R^3$ with 
groups of $SO(3,\R)$ symmetries were given by Jorge and Meeks \cite{JorgeMeeks1983T}, 
Rossman \cite{Rossman1995}, Small \cite{Small1999}, and others. 
Choi, Meeks and White proved in \cite{ChoiMeeksWhite1990}
that if $X$ is a minimal surface in $\R^3$ with a catenoidal end, then every intrinsic 
local isometry of $X$ extends to a rigid motion of $\R^3$. 
As a corollary due to Xu \cite[Corollary 2.2]{Xu1995}, one 
sees that if such $X$ has finite total curvature and embedded ends, at least one of 
which is catenoidal, then the symmetry group of $X$ is a closed subgroup of $SO(3)$.
\end{example}

It is natural to ask which finite groups arise in the context of Theorem \ref{th:main}
for Riemann surfaces of genus $\ggot\ge 1$. 
The study of finite groups $G$ acting effectively on a connected Riemann surface
$X$ by holomorphic automorphisms is based on the observation that 
the orbit space $X/G$ has the structure of a Riemann surface such that the quotient 
projection $\pi:X\to X/G$ is holomorphic, it is ramified precisely at the points 
$x\in X$ with nontrivial stabiliser group $G_x$, 
and the ramification index at such a point equals $k_x=|G_x|$, the order of the stabiliser. 
Furthermore, stabilisers of points in the $G$-orbit of $x$ are conjugate cyclic
subgroups of $G$, so there are $|G|/k_x$ of them 
(see Miranda \cite[Proposition 3.3, p.\ 77]{Miranda1995}).  
By the uniformization theorem for Riemann surfaces, we have that $X=\U/K$ where $\U$ is 
either the Riemann sphere $\CP^1=\C\cup\{\infty\}$, the complex number field $\C$, 
or the upper halfplane
\begin{equation}\label{eq:H}
	\H=\{z=x+\imath y \in\C:y>0\}, 
\end{equation}
and $K$ is a subgroup of $\Aut(\U)$ acting properly discontinuously 
and without fixed points. If $\U=\CP^1$ then $K$ is the trivial group,
and if $\U=\C$ then $K$, if nontrivial,
is a free cyclic group of rank one or two generated by one or two translations.
The case $\U=\H$ is more complicated and will be discussed in Example
\ref{ex:hyperbolic}. Any subgroup $G\subset \Aut(X)$ is then isomorphic
to a quotient group $\Gamma/K$, where $\Gamma$ is a subgroup of 
$\Aut(\U)$ containing $K$ as a normal subgroup. By analysing these conditions, 
the Riemann--Hurwitz formula provides limitations on the 
number and type of finite or discrete groups acting on a given
compact Riemann surface $X$; see Miranda \cite[Chapter III]{Miranda1995}. 
These results also apply to open Riemann surfaces 
of finite genus. Indeed, by Maskit \cite{Maskit1968} every open 
Riemann surface $X$ of finite genus embeds in a compact Riemann surface 
$X^*$ of the same genus such that every holomorphic automorphism 
of $X$ extends to a holomorphic automorphism of $X^*$. 
Applying this technique, Miranda \cite[pp.\ 80--82]{Miranda1995} 
discusses finite subgroups of $\Aut(X)$ for compact Riemann surfaces. 
In the simplest case when $X=\CP^1$, the quotient projection
$\pi:\CP^1\to \CP^1$ has either two or three ramification points. 
The case of two ramification points corresponds to cyclic 
groups of rotations on $\C$. In the case of three ramification points and 
considering $\CP^1$ as the round sphere in $\R^3$, we have the dihedral group 
and the crystallographic groups $A_4,\ S_4$, and $A_5$ 
mentioned in Example \ref{ex:sphere}.

%
%
\begin{example}[{\bf Equivariant minimal surfaces of genus $\ge 2$}]\label{ex:hyperbolic} 
The projective special linear group $PSL(2,\R)=SL(2,\R)/\{ \pm I \}$ 
of degree two over the real numbers can be realised as the 
group of orientation preserving isometries of the hyperbolic plane.
The Poincar\'e halfplane model is given by the upper halfplane 
$\H$ \eqref{eq:H}, endowed with the metric 
$\frac{dx^2+dy^2}{y^2}$ of constant negative curvature,   
on which $PSL(2,\R)$ acts by holomorphic automorphisms
\begin{equation}\label{eq:autH}
	\H\ni z\mapsto \frac{az+b}{cz+d} \quad \text{for}\ a,b,c,d\in\R,\ ad-bc=1.
\end{equation}
This action realises $PSL(2,\R)$ as the holomorphic automorphism group $\Aut(\H)$. 
(One can also use the Poincar\'e disc model $\D=\{z\in \C:|z|<1\}$  
with the Poincar\'e metric $\frac{4|dz|^2}{(1-|z|^2)^2}$
and $PSL(2,\R)$ acting as the group $\Aut(\D)$.) 
A subgroup $\Gamma\subset PSL(2,\R)$ is called a {\em Fuchsian group}
if it acts on $\H$ (by maps \eqref{eq:autH}) properly discontinuously.
General Fuchsian groups were first studied by Poincar\'e  \cite{Poincare1882}, 
who was motivated by Fuchs \cite{FuchsL1880}. 

Every Riemann surface $X$ of genus $\ggot\ge 2$ is a quotient  
$X=\H/K$, where $K\subset \Aut(\H)$ is a Fuchsian group acting without fixed points. 
Every group $G\subset \Aut(X)$ is then of the form $G\cong \Gamma/K$, 
where $\Gamma\subset \Aut(\H)$ is subgroup containing $K$ as a normal subgroup.
If $X$ is compact then $\Gamma$ is of a special form described by Moore in 
\cite[p.\ 923]{MooreMJ1970}. 

For a compact Riemann surface $X$ of genus $\ggot\ge 2$,  
Hurwitz's automorphism theorem \cite{Hurwitz1893} 
(see also \cite[Theorem 3.9, p.\ 96, and Chapter VII]{Miranda1995}) says that
the automorphism group $\Aut(X)$ is finite of order at most $84(\ggot-1)$. 
In view of the aforementioned theorem by Maskit \cite{Maskit1968}, the same
holds on every noncompact Riemann surface of finite genus $\ggot\ge 2$. 
The maximal size $84(\ggot-1)$ can arise if and only if $X$ admits a branched 
cover $X\to \CP^1$ with three ramification points, of indices 2, 3, and 7. 
A group for which the maximum is achieved is called a {\em Hurwitz group}, 
and the corresponding Riemann surface is a {\em Hurwitz surface}.
Klein's quartic curve of genus 3 (see \cite{Klein1879}) is a Hurwitz surface of lowest genus. 
From Klein's result, Macbeath \cite{Macbeath1961} deduced the existence of 
Hurwitz surfaces of infinitely many genuses. The next smallest genus 
of a Hurwitz surface is $\ggot = 7$; see Macbeath \cite{Macbeath1965}
for an explicit description. 
Most Riemann surfaces of genus $\ggot \ge2$ do not have any nontrivial holomorphic 
automorphisms. 
\end{example}

Greenberg \cite{Greenberg1960} proved that every 
countable group $G$ is the automorphism group of a noncompact Riemann surface, 
which can be taken to have a finitely generated fundamental group if $G$ is finite. 
He also proved \cite[Theorem 6']{Greenberg1974} that every finite group is the 
automorphism group of a compact Riemann surface
(see also Jones \cite{GAJones2019}). Greenberg's theorem,
together with Corollary \ref{cor:everyG}, implies the following result.

%
%
\begin{corollary}\label{cor:everyG2}
For every finite group $G$ of order $n>1$ there exist an open connected 
Riemann surface $X$, effective actions of $G$ by holomorphic automorphisms on $X$ 
and by orthogonal transformations on $\R^{2n}$, and a nondegenerate $G$-equivariant
conformal minimal immersion $X\to\R^{2n}$. The surface $X$ can be chosen to be 
the complement of $n$ points in a compact Riemann surface.
\end{corollary}

Since the conditions in Theorem \ref{th:main} pertain to nontrivial
isotropy subgroups of a given automorphism group $G\subset \Aut(X)$, 
it is of interest to understand the possible number of fixed points of 
holomorphic automorphisms of Riemann surfaces. There is a
considerable literature on this subject. Hurwitz \cite{Hurwitz1893} proved 
that every nontrivial holomorphic automorphism of a compact Riemann surface 
of genus $\ggot$ has at most $2\ggot +2$ fixed points.  
In view of the result of Maskit \cite{Maskit1968}, Hurwitz's theorem 
also holds on every open Riemann surface of finite genus.
Moore \cite{MooreMJ1970} determined the number of fixed points of 
each element of a cyclic group of automorphisms of a compact Riemann surface 
with genus at least two. It was shown by Minda \cite[Theorem 1]{Minda1979} that
if $X$ is a Kobayashi hyperbolic Riemann surface and $\phi:X\to X$ is a holomorphic
self-map with at least two fixed points, then $\phi$ is an automorphism 
of $X$ of finite order. 
We refer to the survey in Miranda \cite[Chapter 3]{Miranda1995} for further information
on this topic. 

%
%
\begin{problem}
Let $X$ be an open Riemann surface of genus $\ggot\ge 2$ 
with a nontrivial automorphism group $\Aut(X)$. Which subgroups $G$ of $\Aut(X)$
are symmetry groups of conformal minimal surfaces $X\to \R^n$
for a given $n\ge 3$? (By Corollary \ref{cor:everyG} every such group arises
for $n=2|G|$.) 
\end{problem}

Of particular interest are minimal surfaces of finite total Gaussian curvature.
We discuss this case in Section \ref{sec:FTC} and obtain an  
analogue of Theorem \ref{th:main} for $G$-equivariant minimal surfaces
having some ends of finite total Gaussian curvature; see Problem \ref{prob:FTC} and 
Theorem \ref{th:FTC}. 

%
%
One may also ask whether the Calabi--Yau problem for minimal surfaces 
(see \cite[Chapter 7]{AlarconForstnericLopez2021}
for background and a survey on this problem) 
has an affirmative answer for minimal surfaces with symmetries.
Explicitly, we pose the following problem.

\begin{problem}\label{prob:CY}
Assume the hypotheses of Theorem \ref{th:main}, and let $M$ be a 
compact, smoothly bounded, $G$-invariant domain in the Riemann surface $X$ 
such that no element of $G$ has any fixed point on $bM$.
Does there exist a continuous $G$-equivariant map $F:M\to\R^n$ whose 
restriction to the interior $\mathring M=M\setminus bM$ is a complete conformal 
minimal immersion? 
\end{problem}

An affirmative answer for the trivial group is given by 
\cite[Theorem 1.1]{AlarconDrinovecForstnericLopez2015PLMS}
(see also \cite[Theorem 7.4.1]{AlarconForstnericLopez2021}), where in 
addition the map $F|_{bM}:bM\to\R^n$ is a topological embedding.

Here is another interesting problem.

%
%
\begin{problem}\label{prob:proper}
In the context of Theorem \ref{th:main}, is there a {\em proper} 
nondegenerate $G$-equivariant conformal minimal immersion $X\to \R^n$?
\end{problem}

We expect the answer to be affirmative. Without the equivariance condition,
proper conformal minimal immersions $X\to\R^n$ $(n\ge 3)$ from an arbitrary
open Riemann surface $X$ exist in great abundance; see
Alarc\'on and L\'opez \cite{AlarconLopez2012JDG} for a 
construction of such surfaces in $\R^3$ which project properly to a plane 
$\R^2\subset \R^3$, and \cite[Theorem 3.10.3]{AlarconForstnericLopez2021} 
for any dimension $n\ge 3$. 
In \cite[Section 10.3]{AlarconForstnericLopez2021} the reader can
also find a survey of the history of this subject.

%
%
It seems likely that an analogue of Theorem \ref{th:main} holds for nonorientable conformal 
surfaces; however, the nature of the isotropy groups can be more complicated, 
and the fixed point set of a Euclidean isometry restricted to the surface may contain curves. 
We shall not study this case here. 
Recall that every conformal minimal immersion $X\to \R^n$ from a nonorientable 
open conformal surface $X$ is given by a $\Igot$-invariant conformal minimal immersion 
$\wt X\to \R^n$ from the orientable 2-sheeted cover $\wt X\to X$ whose 
deck transformation is a fixed-point-free antiholomorphic involution $\Igot:\wt X\to \wt X$.
For the theory of such surfaces, see \cite{AlarconLopez2015GT,AlarconForstnericLopezMAMS}.

%
%
The problem treated in this paper can be considered for conformal minimal surfaces in any
Riemannian manifold of dimension $\ge 3$ with a nontrivial group of isometries. 
Although the connection to complex analysis is lost in general, 
there are some other special cases (besides the Euclidean spaces)
which could possibly be approached with these techniques. 
One of them concerns {\em superminimal surfaces} 
in self-dual or anti-self-dual Einstein four-manifolds. This case can be treated using  
the Bryant correspondence in Penrose twistor spaces, thereby 
reducing problems on superminimal surfaces to those concerning
holomorphic Legendrian curves in complex
contact three-manifolds. We refer to the survey of this subject 
in \cite{Forstneric2021JGEA} where
the Calabi--Yau problem (see Problem \ref{prob:CY}) was solved affirmatively
for superminimal surfaces in such Riemannian four-manifolds. (The special 
case concerning the four-sphere with the spherical metric was obtained beforehand 
by Alarc\'on et al. in \cite{AlarconForstnericLarusson2022GT}.)
See also \cite{Forstneric2023AFSTM} for the construction of proper 
superminimal surfaces in the hyperbolic four-space.
 
%
%
Important examples of minimal surfaces are holomorphic curves 
in complex Euclidean spaces $\C^n$ $(n>1)$ and, more generally, in K\"ahler manifolds. 
In principle, equivariant holomorphic maps are easier to construct than 
general minimal surfaces since many more operations and techniques are available. 
Heinzner proved in \cite{Heizner1988} that if $G$ is a reductive 
complex Lie group acting on a reduced Stein space $X$ by holomorphic automorphisms, 
then $X$ is $G$-equivariantly embeddable in a Euclidean space $\C^n$ 
on which $G$ acts by $\C$-linear automorphisms if and only if 
the Luna slice type of $(X,G)$ is finite.  
Heinzner's theorem implies in particular that every open Riemann surface with an 
action of a finite group $G$ of holomorphic automorphisms is equivariantly 
embeddable in some $\C^n$ with a $\C$-linear action of $G$.
See also \cite{Heinzner1989,Heinzner1991}. Further results 
were obtained by Heinzner and Huckleberry \cite{HeiznerHuckleberry1994},
Fritsch and Heinzner \cite{FritschHeinzner2022},
D'Angelo \cite{DAngelo2021}, among many others. In particular,
D'Angelo and Xiao \cite{DAngeloXiao2017} studied equivariant proper 
rational maps between balls in complex Euclidean spaces.
For results on the $G$-equivariant Oka principle, also used 
in this paper, see the survey by 
Kutzschebauch et al.\ \cite{KutzschebauchLarussonSchwarz2022}. 
%

%
%
%
%
\section{Preliminaries}\label{sec:prelim}
Let $X$ be a connected open Riemann surface. An immersion $F:X\to \R^n$ 
is conformal if and only if its $(1,0)$-differential
$\di F=(\di F_1,\ldots, \di F_n)$ satisfies the nullity condition 
\[
	\sum_{i=1}^n (\di F_i)^2=0, 
\]
and it is harmonic if and only if $\di F$ is a holomorphic $1$-form on $X$
(see \cite{Osserman1986} or \cite[Sect.\ 2.3]{AlarconForstnericLopez2021}). 
A conformal immersion is harmonic if and only if its image is a minimal
surface, i.e., its mean curvature vector field vanishes identically.

Pick a nowhere vanishing holomorphic $1$-form $\theta$ on $X$;
such exists by Gunning and Narasimhan \cite{GunningNarasimhan1967}
and can be chosen to be the differential $\theta=dh$ of a holomorphic immersion $h:X\to\C$. 
Then, a conformal minimal immersion $F:X\to \R^n$ satisfies 
$2\di F=f\theta$ with $f=2\di F/\theta :X\to \nullq_*=\nullq\setminus\{0\}$ 
a holomorphic map into the punctured null quadric, where the 
null quadric in $\C^n$ is the affine subvariety
\begin{equation}\label{eq:nullq}
	\nullq=\{z=(z_1,\ldots,z_n)\in \C^n: z_1^2+z_2^2+\cdots+z_n^2=0\}. 
\end{equation}  
These observations lead to the Enneper--Weiestrass formula 
(see \cite[Theorem 2.3.4]{AlarconForstnericLopez2021}), 
which says that any conformal minimal immersion $F:X\to \R^n$ is of the form
\begin{equation}\label{eq:EW}
	F(x)=F(x_0) +  \int_{x_0}^x \Re(f\theta), \quad x,x_0\in X,
\end{equation}
where $f:X\to \nullq_*$ is a holomorphic map such that $\int_C \Re(f\theta)=0$
for any closed path $C$ in $X$ (so the integral in \eqref{eq:EW} is independent of the 
path of integration), and $2\di F=f\theta$.

In the sequel, we shall allow the 1-form $\theta$ on $X$ to have a discrete zero set
and will let $f$ be a meromorphic map 
such that the vector-valued 1-form $f\theta$ is holomorphic and 
nowhere vanishing on $X$, i.e., the poles of $f$ exactly cancel the zeros of $\theta$.
(In the proof of Theorem \ref{th:FTC} we shall allow $f\theta$ to have poles.) 

Let $G$ be a finite group acting faithfully on $X$ by holomorphic automorphisms. 
We have already mentioned that the stabiliser $G_x$ of any point $x\in X$
is a cyclic group of some order $k=k(x)\in\N$   
which is generated in a suitable local holomorphic coordinate $z$ 
on a neighbourhood of $x\in X$, with $z(x)=0$, by the rotation
$z\mapsto \E^{\imath\phi}z$, where
$\E^{\imath\phi}$ is a primitive $k$-th root of $1$. It follows that $z^k$ is
a local holomorphic coordinate on the orbit space $X/G_x$, which is therefore
nonsingular. Note that $G_{gx}=gG_xg^{-1}$ for all $x\in X$ and $g\in G$, so the stabilisers
of points in a $G$-orbit are pairwise conjugate subgroups of $G$.
We can identify the orbit $Gx=\{gx:g\in G\}$ with the set of cosets $\{gG_x:\ g\in G\}$. 
For any $g\in G\setminus\{1\}$ the set of fixed points 
$
	\mathrm{Fix}(g)=\{x\in X:gx=x\}
$ 
is a closed discrete subset of $X$ (see Miranda \cite[Proposition 3.2, p.\ 76]{Miranda1995}), 
which is finite if the surface $X$ has finite genus (see Minda \cite{Minda1979}), but 
it can be infinite otherwise. Their union
\begin{equation}\label{eq:X0}
	X_0= \bigcup_{g\in G\setminus \{1\}} \mathrm{Fix}(g) =\{x\in X: G_x\ne \{1\}\} 
\end{equation}
is a closed, discrete, $G$-invariant subset of $X$, and its complement 
\begin{equation}\label{eq:X1}
	X_1=X\setminus X_0 =\{x\in X: gx\ne x\ \text{for all}\ g\in G\setminus \{1\}\}
\end{equation}
is an open $G$-invariant domain. For every $x\in X$ the orbit $Gx$ has $|G|/|G_x|$ points; 
this number equals $|G|$ if and only if $x\in X_1$. 
Since the group $G$ is finite, the orbit space $X/G$ 
is an open Riemann surface, the quotient projection 
$\pi:X\to X/G$ is a holomorphic map which branches precisely 
at the points in $X_0$, and $\pi:X_1\to X_1/G$ is a holomorphic covering projection
of degree $|G|$. (See Miranda \cite[Theorem 3.4, p.\ 78]{Miranda1995} for these facts.)
Choose a holomorphic immersion $\tilde h:X/G\to \C$; see 
\cite{GunningNarasimhan1967}. Then, the holomorphic map
\begin{equation}\label{eq:h}	
	h:=\tilde h\circ \pi:X\to \C
\end{equation}
is $G$-invariant and it branches precisely at the  points of $X_0$. 
Applying the chain rule to the equation $h\circ g=h$ $(g\in G)$ shows that the 
holomorphic $1$-form 
\begin{equation}\label{eq:theta0}
	\theta=dh=d(\tilde h\circ \pi)=\pi^* d\tilde h
\end{equation}
on $X$ satisfies the following invariance conditions for every $g\in G$:
\begin{equation}\label{eq:theta}
	\theta_{gx} \circ dg_x= \theta_x \ \ \text{for all $x\in X$,\ \ and 
	$\theta_x=0$ if and only if $x\in X_0$}.
\end{equation}
More precisely, $\theta$ has a zero of order $|G_x|-1$ at a point $x\in X_0$. 

%
%
\begin{remark}\label{rem:disconnectedX2}
Everything said so far in this section also holds if $X$ is disconnected and 
for every connected component $X'$ of $X$ the stabiliser subgroup 
$G_{X'}=\{g\in G: gX'=X'\}$ acts effectively on $X'$. 
(This is equivalent to asking that the stabiliser $G_x$
of a generic point $x\in X$ is trivial.) Without loss of generality,
one may always assume that  $G$ acts transitively on the set of connected
components of $X$, so the orbit space $X/G$ is connected. 
These observations can be used to justify the generalization of our results
mentioned in Remark \ref{rem:disconnectedX}. 
\end{remark}

%
%
Suppose now that the group $G$ also acts on $\R^n$ $(n\ge 3)$ by orthogonal maps.
Considering $\R^n$ as the standard real subspace of $\C^n$, the orthogonal
group $O(n,\R)$ is a subgroup of the complex orthogonal group $O(n,\C)$, 
the subgroup of $GL(n,\C)$ preserving the $\C$-bilinear form 
$(z,w)\mapsto \sum_{i=1}^n z_i w_i$. 
The punctured null quadric $\nullq_*=\nullq\setminus \{0\}$ (see \eqref{eq:nullq}) 
is smooth and $O(n,\C)$-invariant, hence also $G$-invariant. 
Consider $\C^n$ as an affine chart in the projective space $\CP^n$.
Let $\overline \nullq\subset \CP^n$ denote the projective closure of $\nullq$ and set 
\begin{eqnarray}
\label{eq:Y}
	Y &=& \overline \nullq\setminus \{0\} = \nullq_*\cup Y_0, \\
\label{eq:Y0}
	Y_0 &=& Y\setminus \nullq_*=\bigl\{[z_1:\cdots:z_n] \in 
	\CP^{n-1}: z_1^2+z_2^2+\cdots+z_n^2=0\bigr\}.
\end{eqnarray}
Let $p:\C^n\setminus\{0\}\to\CP^{n-1}$ denote the projection 
$p(z_1,\ldots,z_n)=[z_1:\cdots:z_n]$. The restriction 
$p:\nullq_*\to Y_0$ is a holomorphic fibre bundle with fibre $\C^*=\C\setminus\{0\}$, 
and the natural extension $p:Y\to Y_0$ which equals the identity map on 
$Y_0$ is a holomorphic line bundle. The action of $O(n,\C)$ on $\C^n$ 
extends to an action on $\CP^n$ with the hyperplane at infinity 
as an invariant complex submanifold. 
Hence, the action of $G$ on $\C^n$ extends to an action of $G$ 
on the manifold $Y$ \eqref{eq:Y} by holomorphic automorphisms. 

We denote by $u\,\cdotp v$ the Euclidean inner product on $\R^n$ 
and by $\|u\|=\sqrt{u\,\cdotp u}$ the Euclidean norm. 
To any oriented 2-plane $0\in \Lambda\subset \R^n$ we associate a complex line
$L\subset \C^n$, contained in the null quadric $\nullq$ \eqref{eq:nullq}, by choosing  
an oriented basis $(u,v)$ of $\Lambda$ such that $\|u\|=\|v\|\ne 0$ and 
$u\,\cdotp v=0$ (such a pair is called a {\em conformal frame}) and setting
\begin{equation}\label{eq:L}
	L=L(\Lambda) = \C(u-\imath v) \subset \nullq\subset \C^n.
\end{equation}
Clearly, $L$ does not depend on the choice of the oriented conformal frame 
on $\Lambda$. A rotation $R_\phi$ on $\Lambda$ in the positive direction 
corresponds to the multiplication by $\E^{\imath \phi}$ on the complex line $L(\Lambda)$. 

If $F:X\to\R^n$ is a conformal immersion then, in any local holomorphic
coordinate $z=x+\imath y$ on $X$, the vectors $\frac{\di F}{\di x}(z)$
and $\frac{\di F}{\di y}(z)$ form a conformal frame and the corresponding
complex line $L(\zeta)\subset \nullq$ is spanned by the vector 
$
	\frac{\di F}{\di x}(z) - \imath \frac{\di F}{\di y}(z) =
	2\frac{\di F}{\di z}(z).  
$

The following proposition summarizes the main properties of immersed $G$-equivariant
conformal minimal surfaces, and it justifies the hypotheses in Theorem \ref{th:main}.

%
%
\begin{proposition}\label{prop:properties}
Assume that $X$ is a connected open Riemann surface and 
$G$ is a finite group acting effectively on $X$ by holomorphic automorphisms 
and acting on $\R^n$ $(n\ge 3)$ by orthogonal transformations. 
Let the sets $X_0\subset X$ and $Y_0\subset Y$ 
be given by \eqref{eq:X0} and \eqref{eq:Y0}, respectively. Set $X_1=X\setminus X_0$, 
and let $\theta$ be the holomorphic 1-form on $X$ given by \eqref{eq:theta0}.
If $F:X\to\R^n$ is a $G$-equivariant conformal minimal immersion then 
\begin{equation}\label{eq:f}	
	f=2\di F/\theta: X\to Y
\end{equation}
is a holomorphic $G$-equivariant map satisfying $f^{-1}(Y_0)=X_0$, and the 
following assertions hold for every point $x_0\in X_0$.
\begin{enumerate}[\rm (a)]
\item The stabiliser $G_{x_0}$ is a cyclic group with a generator $g_0$ 
acting in a local holomorphic coordinate $z$ on $X$, with $z(x_0)=0$, 
by $g_0 z = \E^{\imath \phi}z$, where $\phi=2\pi/k$ and $k=|G_{x_0}|$. 
\item 
The tangent plane $\Lambda=dF_{x_0}(T_{x_0}X) \subset \R^n$ is $G_{x_0}$-invariant, 
$g_0$ acts on $\Lambda$ by rotation $R_\phi$ through the angle $\phi=2\pi/k$ in the positive
direction (with respect to the orientation induced from $T_{x_0}X$ by $dF_{x_0}$), 
and $g_0$ acts on the null line $L=L(\Lambda)$ \eqref{eq:L}
as multiplication by $\E^{\imath \phi}$.
\item
We have that $g_0F(x_0)=F(x_0)$, and the vector $F(x_0)$ is orthogonal to $\Lambda$.
\vspace{1mm}
\item 
We have that $f(x_0)= p(L) \in Y_0\subset \CP^n\setminus \C^n$, 
and $f$ has a pole of order $|G_{x_0}|-1$ at $x_0$.
\end{enumerate}
\end{proposition}

\begin{proof}
Recall that $\di F=(\di F_1,\ldots,\di F_n)$ is a holomorphic $1$-form with values
in $\nullq_*=\nullq\setminus\{0\}$ \eqref{eq:nullq}. 
Since $\theta$ vanishes precisely on $X_0$, the map \eqref{eq:f} is holomorphic 
and satisfies  
\begin{equation}\label{eq:poles}
	f^{-1}(Y_0)=\{x\in X: f(x)\in Y_0\} = X_0.
\end{equation}
Differentiation of the $G$-equivariance equation $F\circ g= g F$ 
(see \eqref{eq:equivariant}), taking into account that $g$ acts on $\R^n$ as 
a linear (orthogonal) transformation, gives
\begin{equation}\label{eq:Fg0}
	dF_{gx}\circ dg_x = g\, dF_x\ \ \text{for every $x\in X$ and $g\in G$}.  
\end{equation}
Writing $dF=\di F+\dibar F$ we have 
\[
	\di F_{gx} \circ dg_x + \dibar F_{gx} \circ dg_x = g\, \di F_x + g\, \dibar F_x.
\]
Since $dg_x$ is $\C$-linear on $T_xX$, the $(1,0)$-part of the above equation gives 
\begin{equation}\label{eq:Fg}
	\di F_{gx}\circ dg_x = g\, \di F_x\ \ \text{for every $x\in X$ and $g\in G$}.  
\end{equation}
From this and the condition \eqref{eq:theta} on the 1-form $\theta$ we obtain
\begin{equation}\label{eq:fg}
	f(gx) = \frac{2\di F_{gx}}{\theta_{gx}} 
		= \frac{2\di F_{gx}\circ dg_x}{\theta_{gx} \circ dg_x}
		= \frac{g\, 2\di F_x}{\theta_x} = g f(x),
\end{equation}
which shows that $f:X\to Y$ is $G$-equivariant.
(The second equality holds since the quotient of two $\C$-linear forms
on $\C$ is invariant under precomposition by a linear isomorphism.) 

Let $k=|G_{x_0}|>1$. By Remark \ref{rem:necessary}, a generator $g_0$ of $G_{x_0}$ 
acts in a certain local holomorphic coordinate $z$ on $X$ based at $x_0$ by the rotation
through the angle $\phi=2\pi/k$, so (a) holds. Since the map 
$dF_{x_0}:T_{x_0}X\cong \R^2 \to \Lambda=dF_{x_0}(\R^2)\subset\R^n$ 
is a conformal linear isomorphism, the relation \eqref{eq:Fg0} at $x=x_0$ implies (b). 
Since $g_0\in \Aut(X)$ fixes $x_0$ and $F$ is $g_0$-equivariant, we have 
$g_0F(x_0)=F(g_0 x_0)=F(x_0)$. As $g_0$ acts on $\R^n$ by 
an orthogonal transformation which restricts to a nontrivial rotation on $\Lambda$,
we infer that $F(x_0)$ is either the zero vector or an eigenvector
of $g_0$ with the eigenvalue $1$ which is orthogonal to $\Lambda$, so (c) holds. 
Since the holomorphic $1$-form $f\theta=2\di F$ has values in $\nullq_*$ and 
$\{\theta=0\}=X_0$, $f$ is a meromorphic map to $\nullq_*$
with poles at the points of $X_0$, and the order of the pole of $f$ 
at $x_0\in X_0$ equals the order of zero of $\theta$ 
at $x_0$, which is $|G_{x_0}|-1$. This proves (d). 
\end{proof}

Conversely, given a connected open Riemann surface $X$,
a holomorphic 1-form on $X$ of the form \eqref{eq:theta0}, and a 
holomorphic map $f:X\to Y$ such that the $1$-form $f\theta$ 
is holomorphic and nowhere vanishing on $X$ and 
\begin{equation}\label{eq:periods}
	\Re \int_C f\theta =0\ \ \text{holds for every smooth  closed curve $C$ in $X$},
\end{equation}
we obtain for any $x_0\in X$ and $v\in \R^n$ a conformal minimal immersion $F:X\to \R^n$
given by  
\begin{equation}\label{eq:EW2}
	F(x)=v + \int_{x_0}^x \Re (f\theta) \ \ \text{for all}\ x \in X. 
\end{equation}
Since $f\theta$ is holomorphic, the integral is independent of the path of integration 
in view of the period vanishing conditions \eqref{eq:periods}. 
If $x_0\in X_0$ then \eqref{eq:gv} implies $gv=v$ for all $g\in G_{x_0}$,
which is compatible with Proposition \ref{prop:properties} (c).  
If on the other hand $x_0\in X\setminus X_0$ then there is no restriction
on $v=F(x_0)\in \R^n$. Let us observe the following.

%
%
\begin{lemma}\label{lem:Gequivariance}
(Assumptions as above.) The conformal minimal immersion
$F:X\to \R^n$, defined by \eqref{eq:EW2}, is $G$-equivariant if and only 
if the map $f:X\to Y$ is $G$-equivariant and 
\begin{equation}\label{eq:gv}
	g v = v + \int_{x_0}^{gx_0} \Re (f\theta) \ \ \text{holds for all $g\in G$.}
\end{equation}
\end{lemma}

\begin{proof} 
Suppose that a map $F:X\to\R^n$ of the form \eqref{eq:EW2} is $G$-equivariant.
The $G$-equivariance condition at the point $x=x_0$ gives
\[
	gv = gF(x_0) = F(gx_0)=v+ \int_{x_0}^{gx_0} \Re(f\theta)
	\quad \text{for all}\ g\in G,
\]
so \eqref{eq:gv} holds. By Proposition \ref{prop:properties} the map $f=2\di F/\theta:X\to Y$ 
is $G$-equivariant as well. 

Conversely, assume that $f:X\to Y$ is a $G$-equivariant holomorphic map. 
Given a piecewise $\Cscr^1$ path $\gamma:[0,1]\to X$, we have in view 
of \eqref{eq:theta} for any $g\in G$ that
\begin{equation}\label{eq:equivarianceofintegral}
	\int_{g\gamma} f\theta =   
		\int_0^1 f(g\gamma(t)) \,\theta_{g\gamma(t)}(dg_{\gamma(t)} \dot\gamma(t)) \,dt 
		= \int_0^1 g f(\gamma(t)) \, \theta_{\gamma(t)} (\dot\gamma(t)) \,dt  
		= g \int_{\gamma} f\theta.
\end{equation} 
If $f$ also satisfies the period vanishing conditions \eqref{eq:periods} 
then the integral of $f\theta$ between a pair of points is independent of the choice
of a path. From \eqref{eq:EW2} and \eqref{eq:gv} we obtain
\begin{eqnarray*}
	F(gx) &=& v+ \int_{x_0}^{gx} \Re (f\theta) 
		 =  v + \int_{x_0}^{gx_0} \Re (f\theta) + \int_{gx_0}^{gx} \Re (f\theta) \\
		 &=& gv + g \int_{x_0}^{x} \Re (f\theta) = gF(x),  
\end{eqnarray*} 
showing that the map $F$ is $G$-equivariant.
\end{proof}

%
%
Summarizing, Proposition \ref{prop:properties} and Lemma \ref{lem:Gequivariance}
give the following representation formula. We explain in Section \ref{sec:infinite}
that the same result holds for infinite discrete groups $G$ acting on $X$ 
properly discontinuously provided that $X/G$ is non-compact.

\begin{theorem}[\rm Weierstrass representation of $G$-equivariant minimal surfaces]
\label{th:GWR}
Assume that $X$ and $G$ are as in Theorem \ref{th:main}, $\theta$
is given by \eqref{eq:h}--\eqref{eq:theta0}, and $Y$ is given by \eqref{eq:Y}.
Fix a point $x_0\in X$. Then, 
every $G$-equivariant conformal minimal immersion $F:X\to\R^n$ is of the form
\[ 
	F(x)=F(x_0) + \int_{x_0}^x \Re (f\theta) \quad \text{for $x\in X$},
\] 
where
$
	f=2\di F/\theta: X\to Y=\nullq_*\cup Y_0
$
is a $G$-equivariant holomorphic map satisfying $f^{-1}(Y_0)=X_0$  
such that $f\theta$ has no zeros or poles and satisfies the period conditions 
\begin{eqnarray*}
	\int_C \Re (f\theta) &=& 0\quad\ \text{for every $[C]\in H_1(X,\Z)$, and} \\
	g F(x_0) &=& F(x_0) + \int_{x_0}^{gx_0} \Re (f\theta) \quad 
	\text{for every $g\in G$.}
\end{eqnarray*}
\end{theorem}

%
%
%
%
\section
{Constructing holomorphic $G$-equivariant maps $X\to Y$}\label{sec:towards}

In this section, we explain the setup and 
outline the proof of Theorem \ref{th:main}; the details
are given in the following two sections.
We shall use the notation from Proposition \ref{prop:properties}. In particular,
the holomorphic $1$-form $\theta$ on $X$ is as in \eqref{eq:theta0} and 
the manifolds $Y_0\subset Y$ are given by \eqref{eq:Y} and \eqref{eq:Y0}.

We begin by defining a $G$-equivariant conformal minimal immersion 
$F_0$ from a neighbourhood of the closed discrete subset
$X_0$ of $X$ (see \eqref{eq:X0}) to $\R^n$. 
Fix a point $x_0\in X_0$ and set $k=|G_{x_0}|>1$, 
where $G_{x_0}$ is the (cyclic) stabiliser group of $x_0$.  
By Proposition \ref{prop:properties} (a), there is a local holomorphic coordinate 
$z$ on a $G_{x_0}$-invariant disc neighbourhood $\Delta\subset X$ of $x_0$, 
with $z(x_0)=0$, such that a generator $g_0$ of $G_{x_0}$ 
is the rotation $g_0 z=\E^{\imath\phi}z$ with $\phi=2\pi/k$. 
By the assumption of Theorem \ref{th:main} there is a $G_{x_0}$-invariant plane 
$0\in\Lambda\subset \R^n$ on which $g_0$ acts as the rotation through the angle 
$\phi$. Let $L=L(\Lambda)$ \eqref{eq:L} be the associated complex line contained
in the null quadric $\nullq\subset \C^n$ \eqref{eq:nullq}. 
Then, $g_0$ acts on $L$ as multiplication by 
$\E^{\imath\phi}$. Choose a nonzero vector $y_0\in L$ and set 
\begin{equation}\label{eq:f0}
	f_0(z) = \frac{y_0}{z^{k-1}} \ \ \text{for all $z\in \Delta$}.
\end{equation}
Thus, $f_0:\Delta\to Y$ is a holomorphic map with 
the point $f_0(x_0)=p(y_0)\in Y_0$ at infinity. Since
$\E^{\imath k \phi}=1$, we have that 
\[
	f_0(g_0z)=f_0(\E^{\imath\phi}z) = \frac{y_0}{\E^{\imath(k-1)\phi} z^{k-1}} 
	= \E^{\imath\phi} \frac{y_0}{\E^{\imath k\phi} z^{k-1}} = g_0 f_0(z),
\]
so $f_0$ is $G_{x_0}$-equivariant. In the coordinate $z$, we have that 
\[
	\theta(z)= d (h(z^k)) =kh'(z^k)z^{k-1}dz, 
\]
where $h:\Delta\to\C$ is a holomorphic function with nonvanishing derivative. 
The 1-form 
\[
	(f_0\theta)(z) = kh'(z^k) y_0 dz
\] 
is holomorphic and nonvanishing on $\Delta$. Since $\Delta$ is simply connected
and condition \eqref{eq:gv} in Lemma \ref{lem:Gequivariance} trivially
holds on $\Delta$ for $v=0$ and the group $G_{x_0}$, $f_0\theta$ integrates
to a flat $G_{x_0}$-equivariant conformal minimal immersion 
$F_0:\Delta\to \Lambda\subset \R^n$. (Alternatively, we can observe that the conformal
linear map $F_0$ from the $z$-coordinate on $\Delta$ to the plane 
$\Lambda$ is $G_{x_0}$-equivariant since the generator of $G_{x_0}$ acts as a 
rotation through the same angle on the domain and
codomain, and take $f_0=2\di F_0/\theta$. This may differ from 
\eqref{eq:f0} by a multiplicative holomorphic factor.)
We extend $f_0$ and $F_0$ by $G$-equivariance to the neighbourhood 
$G\,\cdotp \Delta = \bigcup_{g\in G} g\Delta\subset X$ of the orbit $Gx_0\subset X$. 
Doing the same at every point of $X_0$ and choosing the neighbourhoods
pairwise disjoint yields a $G$-equivariant 
holomorphic map $f_0:V\to Y$ from a $G$-invariant neighbourhood $V\subset X$ of $X_0$
such that $f_0\theta$ is a nowhere vanishing holomorphic $1$-form on $V$
with values in $\nullq_*$, and a $G$-equivariant conformal minimal immersion 
$F_0:V\to\R^n$ with $2\di F_0=f_0\theta$.  

To prove Theorem \ref{th:main}, we shall 
find a $G$-equivariant holomorphic map $f:X\to Y$ which agrees
with $f_0:V\to Y$ to a given finite order in every point of $X_0$, it satisfies
$f(X_1)\subset \nullq_*$ (where $X_1=X\setminus X_0$), 
and conditions \eqref{eq:periods} and \eqref{eq:gv} hold. By Lemma 
\ref{lem:Gequivariance} (or Theorem \ref{th:GWR}), 
the map $F:X\to\R^n$ given by \eqref{eq:EW2} is then 
a $G$-equivariant conformal minimal immersion.

Let us first explain how to find a $G$-equivariant holomorphic map $f:X\to Y$
which agrees with $f_0$ to a given finite order on $X_0$ and satisfies
$f(X_1)\subset \nullq_*$, ignoring the period vanishing 
conditions \eqref{eq:periods} and \eqref{eq:gv} for the moment. 
This is a special case of 
\cite[Theorem 4.1]{KutzschebauchLarussonSchwarz2021JGEA} 
due to Kutzschebauch, L\'arusson, and Schwarz. For our 
purposes, some additional explanations are necessary. 
Consider the action of $G$ on the product manifold $X\times Y$ by 
\begin{equation}\label{eq:gxy}
	g(x,y)=(gx,gy)\ \ \ \text{for}\ \ x\in X,\ y\in Y,\ g\in G.
\end{equation} 
Since the projection $X\times Y\to X$ is $G$-equivariant, it induces a projection 
\begin{equation}\label{eq:Z}
	\rho:Z=(X\times Y)/G \to X/G
\end{equation}
onto the open Riemann surface $X/G$. Note that $Z$ 
is a reduced complex space, the map $\rho$ is holomorphic, it is branched 
over the closed discrete subset $X_0/G$ of $X/G$, and the restriction 
\[
	\rho: Z_1=\rho^{-1}(X_1/G) \to X_1/G
\] 
is a holomorphic $G$-bundle with fibre $Y$. Since the submanifold
$Y_0 \subset Y$ and its complement $Y\setminus Y_0=\nullq_*$
are both $G$-invariant, we have that 
\[
	Z_1 = \left[ (X_1\times \nullq_*)/G\right] \cup \left[(X_1\times Y_0)/G \right]
\]
where the union is disjoint. The open subset 
\begin{equation}\label{eq:Omega}
	\Omega=(X_1\times \nullq_*)/G  \subset Z
\end{equation}
is without singularities, and its complement 
\[
	Z'=Z\setminus \Omega = 
	\left[ (X_0\times Y)/G  \right] \cup \left[ (X_1\times Y_0)/G  \right] 
\]
is a closed complex subvariety of $Z$ containing the branch locus of $\rho$. 
The restricted projection 
\begin{equation}\label{eq:rho} 
	\rho:\Omega \to X_1/G
\end{equation} 
is a holomorphic $G$-bundle with fibre $\nullq_*$, the punctured null quadric.
To describe the structure of this bundle, fix a point $x_1\in X_1$ and let 
$\tilde x_1=\pi(x_1)\in X_1/G$. A loop 
$\gamma:[0,1]\to X_1/G$ with $\gamma(0)=\gamma(1)=\tilde x_1$ lifts with
respect to the covering projection $\pi:X_1\to X_1/G$ to 
a unique path $\lambda:[0,1]\to X_1$ with $\lambda(0)=x_1$. Since the fibres
of $\pi$ are $G$-orbits of the free action of $G$ on $X_1$, its terminal 
point satisfies $\lambda(1)=gx_1$ for a unique $g=g(\gamma)\in G$,  
which only depends on the homotopy class of $\gamma$ in 
the fundamental group $\pi_1(X_1/G,\tilde x_1)$. Conversely,
every $g\in G$ equals $g(\gamma)$
for some loop $\gamma$ in $X_1/G$ based at $\tilde x_1$, and the identity $1\in G$ 
corresponds to loops in the image of the injective homomorphism 
$\pi_*:\pi_1(X_1,x_1)\mapsto \pi_1(X_1/G,\tilde x_1)$ 
induced by the quotient projection $\pi:X_1\to X_1/G$.
In fact, the correspondence $\gamma\mapsto g(\gamma)$ realises 
an isomorphism $\pi_1(X_1/G,\tilde x_1) / \pi_*(\pi_1(X_1,x_1)) \cong G$.
The monodromy homomorphism of the bundle \eqref{eq:rho} along 
the loop $\gamma$ is then given by the action of $g=g(\gamma)$ 
on the fibre $\nullq_*$ of $\rho$ over the point $\tilde x_1$. 
A point $z_1\in \Omega$ with $\rho(z_1)=\tilde x_1$ is represented by a pair 
$(\tilde x_1,y_1)$ for some $y_1\in \nullq_*$, and the monodromy map detemined
by $\gamma$ identifies it with the point $(\tilde x_1,gy_1)$.

%
%
\begin{lemma}\label{lem:correspondence}
There is a natural bijective correspondence between (continuous or holomorphic)
sections $\tilde f:X/G \to Z$ of the map $\rho:Z\to X/G$ \eqref{eq:Z},
satisfying 
\[
	\tilde f(X_1/G)\subset \Omega=(X_1\times \nullq_*)/G
	\ \ \text{and}\ \ 
	\tilde f(X_0/G) \subset (X_0\times Y_0)/G, 
\]
and (continuous or holomorphic) $G$-equivariant maps $f:X\to Y$
satisfying 
\[
	f(X_1)\subset \nullq_*\ \ \text{and}\ \ f(X_0)\subset Y_0.
\]
\end{lemma}

\begin{proof} 
Let us first explain this correspondence over the domain $X_1=X\setminus X_0$.
Given a $G$-equivariant map $f:X_1\to \nullq_*$, we define a section 
$\tilde f: X_1/G \to \Omega$ of the map $\rho$ \eqref{eq:rho} by 
$\tilde f(\pi(x)) = [(x,f(x))]\in \Omega$ for $x\in X_1$,
where $[(x,y)]\in \Omega$ denotes the equivalence class of  
$(x,y)\in X_1\times \nullq_*$ under the action \eqref{eq:gxy}. 
Given $x'\in X_1$ with $\pi(x')=\pi(x)$, we have $x'=gx$ for a unique $g\in G$.
From $f(gx)=gf(x)$ we obtain 
\[
	[(x',f(x'))]=[(gx,f(gx))] = [(gx,gf(x))]=[g(x,f(x))]=[(x,f(x))],
\]
so $\tilde f$ is well-defined. 
Conversely, given a section $\tilde f:X_1/G \to \Omega$, we have 
for every $x\in X_1$ that $\tilde f(\pi(x)) =[(x,y)]\in \Omega$ for a unique $y\in \nullq_*$, 
and we define $f(x)=y$. If $[(x',y')]=[(x,y)]\in \Omega$ then $(x',y')=(gx,gy)$ for some 
$g\in G$, which shows that $f(gx)=f(x') = y'=gy=gf(x)$, so the map $f$ is $G$-equivariant.
The conclusion of the lemma now follows 
by continuity and Riemann's removable singularities theorem.
\end{proof}

Since the holomorphic map $f_0:V\to Y$ from a $G$-invariant neighbourhood 
$V\subset X$ of $X_0$, defined above, 
is $G$-equivariant and satisfies $f_0(V\setminus X_0)\subset \nullq_*$
and $f_0(X_0)\subset Y_0$, Lemma \ref{lem:correspondence} shows that 
$f_0$ determines a holomorphic section $\tilde f_0:V/G \to Z$ such that 
\begin{equation}\label{eq:inclusions}
	\tilde f_0((V\setminus X_0)/G) \subset \Omega \ \ \text{and}\ \ 
	\tilde f_0(X_0/G)\subset (X_0\times Y_0)/G.
\end{equation}
Conversely, we have the following lemma.  

%
%
\begin{lemma}\label{lem:onV}
Let $V\subset X$ be an open 
$G$-invariant neighbourhood of $X_0$ whose connected components are simply 
connected and each of them contains precisely one point of $X_0$, and let
$\Omega$ be given by \eqref{eq:Omega}.  
Assume that $\tilde f_0:V/G \to Z$ is a holomorphic section of the map 
$\rho:Z\to X/G$ \eqref{eq:Z} over $V/G$ satisfying \eqref{eq:inclusions}, 
and let $f_0:V\to Y$ be the associated $G$-equivariant holomorphic 
map satisfying $f_0(V\setminus X_0)\subset \nullq_*$ and $f_0(X_0)\subset Y_0$
(see Lemma \ref{lem:correspondence}). If the $1$-form $f_0\theta$ has no zeros or poles  
on $V$, then $f_0$ determines a $G$-equivariant conformal minimal immersion 
$F_0:V\to\R^n$ with $2\di F_0=f_0 \theta$. 
\end{lemma}

\begin{proof}
Fix a point $x_0\in X_0$. The connected component $\Delta$ of $V$ 
containing $x_0$ is simply connected, $G_{x_0}$-invariant, and the restricted map 
$f_0:\Delta\to Y$ is $G_{x_0}$-equivariant. Choose a vector $v_0\in \R^n$
such that $gv_0=v_0$ for all $g\in G_{x_0}$ (this holds e.g.\ for $v_0=0$)
and define a conformal minimal immersion $F_0:\Delta\to\R^n$ by  
$ 
	F_0(x)=v_0 + \int_{x_0}^x \Re (f_0\theta)
$
for all $x\in \Delta$.  The integral is well-defined since $\Delta$ is simply connected, 
and $F_0$ is $G_{x_0}$-equivariant by Lemma \ref{lem:Gequivariance}. 
(Indeed, condition \eqref{eq:gv} trivially holds for all $g \in G_{x_0}$.) 
We extend $F_0$ by the $G$-equivariance condition to the domain  
$G\, \cdotp \Delta=\bigcup_{g\in G}g\Delta\subset X$. 
The proof is completed by performing the same construction on the other 
$G$-orbits of the closed discrete set $X_0\subset X$.  
\end{proof}

We have now arrived at the main point of our argument. 
The punctured null quadric $\nullq_*$ is a homogeneous
manifold of the complex Lie group $O(n,\C)$ 
\cite[p.\ 78]{AlarconForstnericLopez2021},
hence an Oka manifold by a theorem of Grauert \cite{Grauert1957I}
(see also \cite[Proposition 5.6.1]{Forstneric2017E}). 
It is even algebraically elliptic (see \cite[Proposition 1.15.3]{AlarconForstnericLopez2021}). 
In particular, the map $\rho:\Omega\to X_1/G$ in \eqref{eq:rho} 
is a holomorphic fibre bundle with Oka fibre $\nullq_*$.
Let $\tilde f_0 :V/G\to Z$ be a holomorphic section as in Lemma 
\ref{lem:onV}, satisfying conditions \eqref{eq:inclusions}.
Recall that the branch locus of the holomorphic map $\rho:Z\to X/G$ projects to $X_0/G$.
Since the section $\tilde f_0$ is holomorphic on a neighbourhood of $X_0/G$,
the Oka principle for sections of branched holomorphic maps
(see \cite[Theorem 6.14.6]{Forstneric2017E}) shows that every continuous
section $\tilde f'_0:X/G\to Z$ with $\tilde f(X_1/G)\subset \Omega$, which agrees 
with $\tilde f_0$ on a neighbourhood of $X_0/G$, is homotopic 
to a holomorphic section $\tilde f:X/G\to Z$ through a homotopy of 
sections that are holomorphic on a smaller neighbourhood of $X_0/G$, 
they agree with $\tilde f'_0$ to any given finite order at every point of $X_0$, 
and they map $X_1/G$ to $\Omega$. (The cited result is an improved version of 
\cite[Theorem 2.1]{Forstneric2003FM}. The main addition is that we can control
the range of the resulting holomorphic section $\tilde f:X/G\to Z$, ensuring that 
$\tilde f(X_1/G) \subset \Omega$.) The existence of a continuous extension $\tilde f'_0$
of $\tilde f_0$ with these properties follows from the observation 
that the homotopy type of the open Riemann surface $X/G$ is that of 
a bouquet of circles and the fibre $\nullq_*$ of the $G$-bundle \eqref{eq:rho} is connected.
(This is a special case of \cite[Corollary 5.14.2]{Forstneric2017E}.)

To prove Theorem \ref{th:main}, we must explain how to find a holomorphic section 
$\tilde f:X/G\to Z$ as above such that the associated $G$-equivariant holomorphic map 
$f:X\to Y$, given by Lemma \ref{lem:correspondence}, 
integrates to a $G$-equivariant conformal minimal immersion $F:X\to \R^n$ 
as in \eqref{eq:EW2}. This amounts to showing that $f$ can be chosen such that
it satisfies the period conditions \eqref{eq:periods} and \eqref{eq:gv}.
We shall follow \cite[proof of Theorem 3.6.1]{AlarconForstnericLopez2021} 
with appropriate modifications to ensure $G$-equivariance. 
The main lemma is given in the following section 
(see Lemma \ref{lem:main}), and Theorem \ref{th:main} is proved in
Section \ref{sec:proof} as a special case of Theorem \ref{th:mainbis}.

%
%
%
%
\section{The main lemma}\label{sec:mainlemma}
The main result of this section is Lemma \ref{lem:main}, which provides the key
step in the proof of Theorems \ref{th:main} and \ref{th:mainbis}. 
This lemma is a $G$-equivariant analogue
of \cite[Proposition 3.3.2]{AlarconForstnericLopez2021}.

We begin by adjusting the relevant technical tools from 
\cite[Chapter 3]{AlarconForstnericLopez2021} to $G$-equivariant minimal surfaces. 
Recall the following notion; see \cite[Definition 1.12.9]{AlarconForstnericLopez2021}.

%
%
\begin{definition}\label{def:admissible}
An {\em admissible set} in a Riemann surface $X$ is 
a compact set of the form $S=K\cup E$, where $K$ is a 
(possibly empty) finite union of pairwise disjoint compact domains with piecewise 
$\Cscr^1$ boundaries in $X$ and $E = S \setminus\mathring  K$ is a union of finitely many 
pairwise disjoint, smoothly embedded Jordan arcs and closed Jordan curves meeting $K$ 
only at their endpoints (if at all) and such that their intersections with the boundary 
$bK$ of $K$ are transverse. 
\end{definition}

Note that $\mathring S=\mathring K$. Since an admissible set $S$ has at most finitely 
many holes in $X$, the Bishop--Mergelyan theorem 
(see Bishop \cite{Bishop1958PJM} and 
\cite[Theorem 5]{FornaessForstnericWold2020}) shows that 
every function in the algebra $\Ascr(S)=\Cscr(S)\cap \Oscr(\mathring S)$ 
is a uniform limit of meromorphic functions on $X$ with poles in $X\setminus S$, and of 
holomorphic functions on $X$ if $X\setminus S$ has no holes (i.e., 
$X$ is an open Riemann surface and $S$ is Runge in $X$). 
Furthermore, functions of class $\Ascr^r(S)$ for any $r\in\N$ can be approximated
in the $\Cscr^r(S)$ topology by meromorphic or holomorphic functions on $X$,
respectively (see \cite[Theorem 16]{FornaessForstnericWold2020}).

We assume in the sequel that $X$ is an open Riemann surface, not necessarily
connected, and $G$ is a finite group acting on $X$ and on $\R^n$ 
as in Theorem \ref{th:main} such that the action of $G$ is transitive on the
set of connected components of $X$ and the stabiliser of any component 
$X'$ of $X$ acts effectively on $X'$ 
(see Remarks \ref{rem:disconnectedX} and \ref{rem:disconnectedX2}).  
These conditions imply that the set $X_0$ \eqref{eq:X0}, which is the union of 
fixed point sets of all elements of $G$, is a closed, discrete, $G$-invariant subset of $X$, 
which is the branch locus of the holomorphic quotient projection 
\[
	\pi:X\to \wt X:=X/G,
\] 
and the orbit space $\wt X$ is a connected Riemann surface.
Set $\wt X_0=X_0/G$.
If $S$ is a $G$-invariant admissible subset of $X$ such that 
$X_0\cap S\subset \mathring S$, then $\wt S=\pi(S)$ is an admissible subset of $\wt X$.
Conversely, given an admissible subset $\wt S\subset \wt X$ such that 
$\wt X_0\cap \wt S$ is contained in the interior of ${\wt S}$, its preimage $S=\pi^{-1}(\wt S)$ 
is a $G$-invariant admissible subset of $X$ such that $X_0\cap S\subset\mathring S$. 

%
%
\begin{lemma}\label{lem:Runge}
Let $\pi:X\to \wt X=X/G$ be as above. 
A compact set $\wt S\subset \wt X$ is Runge in $\wt X$ if and only if
its preimage $S=\pi^{-1}(\wt S) \subset X$ is Runge in $X$. 
\end{lemma}

\begin{proof}
Recall that a compact set $S$ in an open Riemann surface $X$ is Runge if
and only if it is $\Oscr(X)$-convex.
Assume that $\wt S$ is Runge in $\wt X$. Given a point 
$p\in X\setminus S$, we have that $\tilde p=\pi(p)\in \wt X\setminus \wt S$, 
and hence there is a holomorphic function $\tilde h\in\Oscr(\wt X)$ satisfying
$\tilde h(\tilde p)=1>\sup_{\tilde x\in \wt S} |\tilde h(\tilde x)|$. 
The function $h=\tilde h\circ \pi \in\Oscr(X)$ then satisfies
$h(p)=1>\sup_{x\in S} |h(x)|$, showing that $S$ is Runge in $X$.
Conversely, assume that $S$ is Runge in $X$. Pick a point
$\tilde p\in  \wt X\setminus \wt S$ and let 
$\pi^{-1}(\tilde p)=\{p_1,\ldots,p_m\} \subset X\setminus S$, where the
fibre points $p_i$ are listed according to their multiplicities. 
Since $S$ is Runge in $X$, there exists $h_0\in\Oscr(X)$
satisfying $\sup_{x\in S}|h_0(x)|<1$ and $h_0(p_i)=1$ for $i=1,\ldots, m$.
The function $h\in\Oscr(X)$ defined by $h(x)=\frac{1}{m}\sum_{g\in G} h_0(gx)$
is then $G$-invariant and satisfies $h(p_i)=1>\sup_{x\in S}|h(x)|$
for $i=1,\ldots, m$. It follows that $h=\tilde h\circ\pi$ where $\tilde h\in\Oscr(\wt X)$ 
satisfies $\tilde h(\tilde p)=1>\sup_{\tilde x\in \wt S} |\tilde h(\tilde x)|$.
Hence, $\wt S$ is Runge in $\wt X$. 
\end{proof}

The following is a version of \cite[Definition 3.1.2]{AlarconForstnericLopez2021},
allowing the 1-form $\theta$ to have zeros.

%
%
%
\begin{definition} \label{def:GCMI}
Let $S=K\cup E$ be an admissible set in a Riemann surface $X$ 
(see Definition \ref{def:admissible}), let $\theta$ be a holomorphic $1$-form 
on a neighbourhood of $S$ in $X$ without zeros on $bS$,
and let $Y$ be the manifold \eqref{eq:Y}. 
A {\em generalized conformal minimal immersion} $S\to\R^n$ $(n\ge 3)$ of class 
$\Cscr^r$ $(r\in\N)$ is a pair $(F,f\theta)$, where $F: S\to \R^n$ is a $\Cscr^r$ map 
whose restriction to $\mathring S$ is a conformal minimal immersion 
and the map $f\in \Ascr^{r-1}(S,Y)$ satisfies the following conditions:
\begin{enumerate}[\rm (a)]
\item 
$f\theta =2\di F$ holds on $K$ (in particular, the zeros of $\theta$ 
cancel the poles of $f$), and
\item 
for any smooth path $\alpha:[0,1]\to X$ parameterizing a connected component of
$E=\overline{S\setminus K}$ we have that 
$\Re(\alpha^*(f\theta))=\alpha^* dF = d(F\circ \alpha)$.
(See Remark \ref{rem:GCMI}.) 
\end{enumerate}
\end{definition}

Given an admissible set $S\subset X$ and integers $r\ge 1,\ n\ge 3$, we denote by 
$
	\GCMI^r(S,\R^n)
$ 
the space of all generalized conformal minimal immersions $S\to\R^n$ of class $\Cscr^r$.  
An element $(F,f\theta) \in \GCMI^r(S,\R^n)$ is said to be {\em nonflat} if the 
image by $F$ of any connected component of $K$ and of $E$ is not contained
in an affine 2-plane in $\R^n$. This holds if and only if the image of any such component
by $f$ is not contained in a ray of $\nullq$, compactified with the point at infinity. 
The identity principle shows that if $F$ is nonflat on a connected 
domain $D$ in a Riemann surface $X$ then its restriction to every arc in 
$D$ is nonflat; conversely, if $F$ is nonflat on a nontrivial arc
then it is nonflat on every connected domain containing this arc.

If $G$ is a finite group as in Theorem \ref{th:main}
and $S$ is $G$-invariant, then $(F,f\theta) \in \GCMI^r(S,\R^n)$ 
is said to be {\em $G$-equivariant} if $F(gx)=gF(x)$ holds for all $x\in S$ and $g\in G$. 
In this case, the map $f:S\to Y$ is also $G$-equivariant (see the 
proof of Proposition \ref{prop:properties}). We denote by 
\[
	\GCMI^r_G(S,\R^n)
\] 
the space of $G$-equivariant generalized conformal minimal immersions 
$S\to\R^n$ of class $\Cscr^r$.  

%
%
\begin{remark}\label{rem:GCMI}
Let $S=K\cup E$ be an admissible set and $(F,f\theta) \in \GCMI^r(S,\R^n)$.
Since $F$ is a conformal minimal immersion on $\mathring K=\mathring S$,
we have that $dF=\Re(2\di F)$ on $K$, and hence 
condition (b) in Definition \ref{def:GCMI} is compatible with condition (a) 
asking that $f\theta=2\di F$ hold on $K$. A map $f:S\to Y$ of class $\Ascr^{r-1}(S)$ 
determines a generalized conformal minimal immersion $(F,f\theta)\in \GCMI^r(S,\R^n)$ 
if and only if $\Re \int_\lambda f\theta=0$ on every closed piecewise 
$\Cscr^1$ path $\lambda \subset S$, and it suffices to verify this condition
on a basis of the homology group $H_1(S,\Z)$. (This is a free abelian group of finite
rank. We refer to \cite[Lemma 1.12.10]{AlarconForstnericLopez2021} for 
the construction of a homology basis with suitable properties that will be used
in the sequel.) In particular, if $F:K\to \R^n$ is a conformal
minimal immersion satisfying condition (a) and $E$ is an oriented arc attached with 
both endpoints $p,q$ to $K$, then $f|_E$ must satisfy the condition 
$\Re \int_E f\theta = F(q)-F(p)$. 
\end{remark}

The following lemma provides the key ingredient in the proof of Theorem \ref{th:main}. 

%
%
\begin{lemma}\label{lem:main}
Let $X$ be an open Riemann surface, $G$ be a finite group as in 
Theorem \ref{th:main} acting transitively on the set of connected components of $X$, 
$X_0$ be the set \eqref{eq:X0}, 
$\theta$ be a holomorphic 1-form on $X$ as in \eqref{eq:theta} with $\{\theta=0\}=X_0$, 
and let $\pi:X\to X/G=\wt X$ denote the quotient projection.
Assume that $S$ is $G$-invariant admissible set in $X$ such that $X_0\subset \mathring S$
and the admissible set $\wt S=\pi(S)\subset \wt X$ is a strong deformation retract of 
$\wt X$. Then, every nonflat $G$-equivariant generalized conformal minimal immersion 
$(F_0,f_0\theta)\in \GCMI^r_G(S,\R^n)$ $(r\in \N)$ 
can be approximated in the $\Cscr^r$ topology on $S$ by nonflat $G$-equivariant 
conformal minimal immersions $F:X\to\R^n$. Furthermore, $F$ can be chosen 
to agree with $F_0$ in any given finite set of points $A\subset K$, 
and to agree with $F_0$ to any given finite order in the points of $A\cap \mathring K$.
\end{lemma}

\begin{proof} 
We shall follow the construction in \cite[Sections 3.2--3.3]{AlarconForstnericLopez2021},
adjusting it to the $G$-equivariant case by using the approach described in Section
\ref{sec:towards}. 

Write $S=K\cup E$ as in Definition \ref{def:admissible}. Since $S$ is $G$-invariant,
both $K$ and $E$ are $G$-invariant.
The conditions imply that the set $X_0$ is finite and the orbit space $\wt X=X/G$ is connected. Since the admissible set $\wt S=\pi(S)\subset \wt X$ is a 
strong deformation retract of $\wt X$, it is connected as well. 
Note that $\wt S$ contains the finite set $\wt X_0=X_0/G$ in its interior. 
Since $\pi:X\setminus X_0 \to \wt X\setminus \wt X_0$ is an unramified covering
projection, it follows that $S$ is a strong deformation retract of $X$. We enlarge
the set $A\subset S$ in the lemma so that it contains $X_0$ 
and the endpoints of all connected components of $E$, and set 
$\wt A=\pi(A)\subset \wt S$. By \cite[Lemma 1.12.10]{AlarconForstnericLopez2021} 
(or \cite[Lemma 3.1]{Forstneric2022APDE})
and \cite[proof of Proposition 3.3.2, p.\ 142]{AlarconForstnericLopez2021} 
there is a finite collection of smooth embedded compact arcs 
$\wt \Cscr=\{\wt C_1,\ldots,\wt C_l\}$ in $\wt S$ 
(diffeomorphic images of $[0,1]\subset\R$) with the following properties.
\begin{enumerate}[\rm (i)]
\item The intersection of any two distinct arcs in $\wt \Cscr$ is either empty
or a common endpoint of both arcs.
\item 
Every point of $\wt A$ is an endpoint of an arc in $\wt \Cscr$. 
\item Every point $\tilde x_0\in \wt X_0$ is an endpoint of a single arc $\wt C_i\in \wt \Cscr$, and the other endpoint of $\wt C_i$ does not belong to $\wt X_0$. 
\item The compact set $\wt C=\bigcup_{i=1}^l \wt C_i$ is 
a strong deformation retract of $\wt S$ (and hence of $\wt X$). In particular, 
$\wt C$ is connected and contains a homology basis of $\wt S$ (and hence of $\wt X$). 
\end{enumerate}
Note that (i) and (ii) imply that no point of $\wt A$ is an interior point of an arc in $\wt \Cscr$.
We now enlarge the finite sets $\wt A$ and $A$ if necessary to also arrange that 
\begin{enumerate}[\rm (v)]
\item The set of all endpoints of the arcs in $\wt \Cscr$ equals $\wt A$, and 
$A=\pi^{-1}(\wt A)$.
\end{enumerate}
For each $i=1,\ldots,l$ the preimage $C_i=\pi^{-1}(\wt C_i)\subset S$ is the union 
of arcs $C_{i,1},\ldots,C_{i,m}$, where $m=|G|$ is the degree of the projection 
$\pi:X\to\wt X$. If an arc $\wt C_i\in \wt \Cscr$ does not contain any point of $\wt X_0$ 
then $\pi$ is a trivial covering projection over $\wt C_i$, and hence the arcs $C_{i,j}$ for $j=1,\ldots,m$ are pairwise disjoint. In the opposite case, one of the endpoints
$\tilde x_0$ of $\wt C_i$ belongs to $\wt X_0$, and then several arcs $C_{i,j}\subset C_i$ 
share an endpoint 
but are otherwise disjoint. In fact, if $k=|G_{x_0}|>1$ is the order of the stabiliser group of 
$x_0\in X_0$ and $\tilde x_0 = \pi(x_0)$, then $x_0$ is the common endpoint of precisely 
$k$ arcs $C_{i,j}\subset C_i$. The set 
\begin{equation}\label{eq:C}  
	C=\pi^{-1}(\wt C)=\bigcup_{i=1}^l C_i
\end{equation} 
is $G$-invariant, Runge in $X$ by Lemma \ref{lem:Runge}, and it contains the set $X_0$ 
of branch points of $\pi$. Since $\wt C$ is a strong deformation
retract of $\wt X$, $\wt X_0\subset \wt C$ and $\pi$ is unbranched 
over $\wt X\setminus \wt X_0$, it follows that $C$ is a strong deformation retract of $X$. 
Our assumptions imply that none of the arcs $f_0(C_{i,j})$ is contained 
in a compactified ray of the null quadric $\nullq_*$.
Hence, \cite[Lemma 3.2.1]{AlarconForstnericLopez2021} yields for every $i=1,\ldots,l$ a 
$\Cscr^{r-1}$ map $h_i:C_{i,1}\times B\to Y=\nullq_*\cup Y_0$ (see \eqref{eq:Y}),
where $B\subset \C^n$ is a ball centred at the origin, with the following properties. 
\begin{enumerate}[\rm (a)]
\item $h_i(x,0)=f_0(x)$ for every $x\in C_{i,1}$.
\item There is a closed arc $I_i$ contained in the relative interior of $C_{i,1}$ such that 
$h_i(x,t)=f_0(x)$ for every $x\in C_{i,1}\setminus I_i$ and $t\in B$.
\item $h_i(x,t)\in \nullq_*$ for every $x\in C_{i,1}\setminus X_0$ and $t\in B$.
\item The map $h_i(x,\cdotp):B\to Y$ is holomorphic for every $x\in C_{i,1}$.
\item The following $\C$-linear map is an isomorphism:
\begin{equation}\label{eq:period-i}
	\frac{\di}{\di t}\Big|_{t=0} \int_{C_{i,1}} h_i(\cdotp,t)\theta : \C^n\to \C^n. 
\end{equation}
\end{enumerate}
In the language of \cite[Chapter 3]{AlarconForstnericLopez2021}, $h_i$ is
a period dominating spray of maps $C_{i,1} \to Y$ with the core $f_0$. 
We now extend $h_i$ by $G$-equivariance to a spray $C_i\times B\to Y$ 
over $C_i=\bigcup_{j=1}^m C_{i,j}$. Recall that 
$C$ is given by \eqref{eq:C}. By condition (b) above, $h_i$ extends 
to a spray $h_i:C \times B\to Y$ which is independent 
of $t\in B$ on $C\setminus C_i$, so it equals $f_0$ there. 
Let $B^l$ denote the Cartesian product of $l$ copies of $B$.
Finally, we combine the sprays $h_1,\ldots, h_l$ into a single  
spray $h:C \times B^l \to Y$, which is $G$-equivariant 
with respect to the first variable $x\in C$, such that, writing the parameter variable
as $t=(t^1,\ldots,t^l)\in B^l$ with $t^i =(t_{i,1},\ldots,t_{i,n})\in B$, we have 
\begin{equation}\label{eq:spray-h}
	h(x,t^1,\ldots,t^l)=h_i(x,t^i)\ \ \text{for all $x\in C_i$ and $i=1,\ldots,l$}.
\end{equation}

To the collection of arcs $\Cscr=\{C_{i,j}\}$ defined above and to any 
continuous $G$-equivariant map $f:C\to Y$ such that $f\theta$ assumes values in 
the punctured null quadric $\nullq_*$ we associate the period vector 
$\Pcal(f) = \big(\Pcal_1(f),\ldots, \Pcal_l(f)\big) \in  (\C^n)^l$
whose $i$-th component equals 
\begin{equation}\label{eq:periodmap}
	\Pcal_i(f) = \int_{C_{i,1}} f\theta \in \C^n\ \ \text{for}\ i=1,\ldots,l.
\end{equation}
Our construction clearly implies that
\[ 
	\frac{\di}{\di t}\Big|_{t=0} \Pcal(h(\cdotp,t)) : (\C^n)^l \to (\C^n)^l
	\ \ \text{is an isomorphism}.
\] 
Indeed, the above linear map has a block structure whose
$i$-th diagonal $n\times n$ block equals the map in \eqref{eq:period-i} while the 
off-diagonal blocks vanish. 

In the next step, we approximate $h$ by a spray of maps 
$H:S\times B^l\to Y$ of class $\Ascr^{r-1}(S\times B^l)$ 
(where the ball $B\subset \C^n$ is allowed to shrink a little) 
such that $H(\cdotp,0)=f_0$, $H(\cdotp,t)$ agrees with 
$f_0$ in every point of $A$ and to any given finite order in points of
$A\cap \mathring K$ (in particular, in points of the set
$X_0\subset \mathring K$), $H$ is
$G$-equivariant in $x\in S$ for any fixed $t\in B^l$, and 
\begin{equation}\label{eq:Hdominating}
	\frac{\di}{\di t}\Big|_{t=0} \Pcal(H(\cdotp,t)) : (\C^n)^l \to (\C^n)^l
	\ \ \text{is an isomorphism}.
\end{equation}
In other words, $H$ is a $G$-equivariant $\Pcal$-period dominating spray on $S$ with 
values in $Y$ and with the core $f_0$. To find such $H$, we proceed as follows. 

By the discussion in Section \ref{sec:towards}, we can view the spray $h$ in
\eqref{eq:spray-h} as a spray of sections $\tilde h: \wt C\times B^l \to Z$ 
of the map $\rho:Z=(X\times Y)/G \to X/G=\wt X$ in \eqref{eq:Z}, whose core
$\tilde h(\cdotp,0):\wt C\to Z$ is the section $\tilde f_0$ restricted to $\wt C$. 
Note that $\tilde h$ is holomorphic in the parameter $t\in B^l$, and 
$\tilde h(\tilde x,t)=\tilde f_0(\tilde x)$ holds for all $t\in B^l$ 
when $\tilde x\in \wt C$ is sufficiently near $\wt A=\pi(A)$. 

By Poletsky \cite[Theorem 3.1]{Poletsky2013} (see also 
\cite[Theorem 32 and Corollary 7]{FornaessForstnericWold2020}), the image
of the section $\tilde f_0:\wt S\to Z$ has an open Stein neighbourhood
$Z_0\subset Z$. (Poletsky's result is stated for sections of holomorphic submersions,
but in the case at hand, the branch points of $\rho:Z\to \wt X$
lie in the interior of $\wt S$ and a minor modification of his proof applies.
For Stein subvarieties without boundaries, the existence of open Stein neighbourhoods
was proved by Siu \cite{Siu1976}.)
By \cite[Proposition 2.2]{Forstneric2003FM} the projection $\rho:Z\to X/G=\wt X$ 
admits a holomorphic fibre-spray over the Stein domain $Z_0\subset Z$,  
which is fibre dominating outside the branch locus of $\rho$, 
that is, on $Z_0\setminus \rho^{-1}(\wt X_0)$, 
and which is trivial over $\rho^{-1}(\wt X_0)$. By restricting this spray 
to $\tilde f_0(\wt S)\subset Z_0$ we obtain a fibre-spray of sections 
$s:\wt S\times W\to Z$, where $0\in W\subset \C^k$
is a neighbourhood of the origin in some complex Euclidean space, 
such that $s(\cdotp,0)=\tilde f_0$, 
$\rho(s(\tilde x,\zeta))=\tilde x$ for all $\tilde x\in \wt S$ and $\zeta\in W$, 
and $s$ is fibre dominating over $\wt S\setminus \wt X_0$.
The fibre domination property of $s$ and the implicit function theorem
imply that we can factor the spray $\tilde h$, constructed above, through the spray $s$.
Explicitly, shrinking the ball $B\subset \C^n$ in the domain of $\tilde h$ if necessary, 
there is a map $\zeta:\wt C\times B^l\to W$ of class $\Ascr^{r-1}$ such that
\[
	\tilde h(\tilde x,t) = s\bigl(\tilde x, \zeta(\tilde x,t)\bigl) \ \ 
		\text{and}\ \ \zeta(\tilde x,0)=0\ \ \text{hold for all} \ \ 
		\tilde x\in \wt C\ \ \text{and} \ \ t\in B^l.
\]
Note that $\zeta$ can be chosen such that 
$\zeta(\tilde x,t)$ vanishes when the point $\tilde x\in \tilde S$ 
is sufficiently close to a point of $\wt A=\pi(A)$. This is because the
spray $\tilde h$ is supported (i.e., not identically equal to its core $\tilde f_0$)
on a union of closed arcs in $\wt C$ disjoint from $\wt A$, 
and over this set we can find a (necessarily) trivial complex vector 
subbundle of $\wt C\times \C^k$ 
which is mapped by the differential $ds$ isomorphically onto the vertical (with respect to 
the map $\rho:Z\to \wt X$) tangent bundle of $Z$, that is, the kernel of $d\rho$.
Furthermore, by the construction we have that $\tilde h(\tilde x,t)=\tilde f_0(\tilde x)$ 
for every point $\tilde x\in \tilde S$ that is sufficiently close to a point 
of $\wt A$ and for every $t\in B^l$.

Consider the Taylor series expansion of $\zeta$ in the $t$-variable:
\[
	\zeta(\tilde x,t)=\sum t_{i,j} \zeta_{i,j}(\tilde x) + O(|t|^2), 
\]
where the summation is over $i=1,\ldots,l$ and $j=1,\ldots,n$ and the 
coefficient functions $\zeta_{i,j}$ are of class $\Cscr^{r-1}(\wt C)$. 
Since $\wt C$ is a piecewise smooth curve which is Runge in $\wt X$, 
Mergelyan's theorem \cite[Theorem 16]{FornaessForstnericWold2020} 
allows us to approximate the functions $\zeta_{i,j}$ in the $\Cscr^{r-1}(\wt C)$ topology 
by holomorphic functions $\tilde \zeta_{i,j}$ on $\wt X$ that vanish 
to any given finite order in the points of $\wt A$.
Consider the map $\wt H:\wt S\times B^l\to Z$ of class $\Ascr^{r-1}$ defined by
\begin{equation}\label{eq:wtH}
	\wt H(\tilde x,t) = s\left(\tilde x, \sum t_{i,j} \tilde \zeta_{i,j}(\tilde x)\right)
	\ \ \text{for $\tilde x\in \wt S$ and $t\in B^l$}.
\end{equation}
Note that $\wt H(\cdotp,0)=\tilde f_0$, and the partial differential $\di_t \wt H|_{t=0}$
is close to $\di_t \tilde h|_{t=0}$ on $\wt C$. Assuming that the 
approximation is close enough, the map $H:S\times B^l\to Y$ determined by $\wt H$
(see Lemma \ref{lem:correspondence}) is a $G$-equivariant spray of class $\Ascr^{r-1}$
with the core $f_0$ which is $\Pcal$-period dominating, 
i.e., \eqref{eq:Hdominating} holds. Indeed, the period domination condition
only depends on the $t$-derivative of the spray at $t=0$ and is stable under deformations.

We can now complete the proof of Lemma \ref{lem:main}. 
Recall that $X_1=X\setminus X_0$, see \eqref{eq:X1}.
Let $\Omega=(X_1\times \nullq_*)/G \subset Z$ be the domain in \eqref{eq:Omega}.
Its complement $Z'=Z\setminus \Omega$ is a closed complex subvariety of $Z$ containing the branch locus of $\rho$, and the restricted projection 
$\rho:\Omega \to \wt X_1:=X_1/G=\wt X\setminus \wt X_0$ is a holomorphic $G$-bundle 
whose fibre $\nullq_*$ is an Oka manifold. Note that the range of the spray 
$\wt H$ in \eqref{eq:wtH} lies in $\Omega$ except over the points $\tilde x\in \wt X_0$,
and these points are contained in the interior of $\wt S$. Hence, shrinking 
the ball $B\subset \C^n$ slightly, we can apply the Oka principle
for sections of branched holomorphic maps
\cite[Theorem 2.1]{Forstneric2003FM} (see also the improved version in 
\cite[Theorem 6.14.6]{Forstneric2017E} which exactly fits our situation) 
to approximate $\wt H$ on $\wt S\times B^l$ by a holomorphic 
spray of sections $\wt \Theta:\wt X\times B^l \to Z$ which agrees with $\wt H$ 
to a given finite order in the points $\tilde x\in \wt A$ and maps 
$\wt X_1 \times B^l$ to $\Omega$. (Recall that $\wt X_1=\wt X\setminus \wt X_0$.
Although \cite[Theorem 6.14.6]{Forstneric2017E} is stated for a single section,
its proof applies to sprays of sections. Alternatively, one can treat the 
parameter variable $t$ as a space variable, suitably adjusting the spaces under consideration.)

Let $\Theta:X\times B^l \to Y$ be the $G$-equivariant holomorphic spray associated 
to $\wt \Theta$ by Lemma \ref{lem:correspondence}. 
Assuming that the approximation of $\wt H$ by $\wt \Theta$
is close enough, the period domination property of $H$ and the 
implicit function theorem yield a parameter value $t_0\in B^l$ close to the origin
such that the $G$-equivariant map $f=\Theta(\cdotp,t_0):X\to Y$ approximates
$f_0$ on $S$ and satisfies $\Pcal(f) = \Pcal(f_0)$; see \eqref{eq:periodmap}.
From the last condition and $G$-equivariance it follows that $f$ has the 
same periods as $f_0$ on the arcs 
$C_{i,j}\subset C$ for $i=1,\ldots,l$ and $j=1,\ldots,m$.  
Furthermore, the construction implies that $f\theta$ is a holomorphic $1$-form on $X$
with values in $\nullq_*$. 

Since the $G$-equivariant set $C$ is a strong deformation retract of $X$,
it contains a homology basis of $X$ and also curves 
which can be used to verify conditions \eqref{eq:gv} 
in Lemma \ref{lem:Gequivariance}. More precisely,
each of these curves can be chosen to be a union of arcs in the family
$\Cscr=\{C_{i,j}\}$. Since the periods of $f\theta$ and $f_0\theta$ 
agree on every arc $C_{i,j}\in \Cscr$ and $f_0$ comes from a generalized 
$G$-equivariant conformal minimal immersion $(F_0,f_0\theta)$, 
it follows that $f\theta$ integrates to a $G$-equivariant conformal 
minimal immersion $F:X\to\R^n$ satisfying Lemma \ref{lem:main}. 
If $X$ is disconnected and the conditions in Remark \ref{rem:disconnectedX}
hold, we apply integration on a connected component of $X$ and 
extend it to all of $X$ by $G$-equivariance.
\end{proof}

In the proof of Theorem \ref{th:main} we shall also need the following lemma
on approximating flat $G$-equivariant conformal minimal immersions
by nonflat ones.

%
%
\begin{lemma}\label{lem:nonflat}
Let $G=\langle g\rangle$ be a cyclic group of order $k$ whose generator $g$ acts on $\C$  
by the rotation $gz=\E^{\imath 2\pi/k}z$, and it acts on $\R^n$ $(n\ge 3)$ 
by an orthogonal transformation. Let $U \subset \C$ be a disc centred at 
the origin. Given a flat $G$-equivariant conformal minimal immersion $F_0:U\to \R^n$ 
and an integer $r\in\N$, there is a nonflat $G$-equivariant conformal minimal immersion 
$F:\C\to\R^n$ which agrees with $F_0$ to order $r$ in $z=0$ and it approximates
$F_0$ as closely as desired on a given compact subset of $U$.
\end{lemma}

\begin{proof}
Assume that $F_0$ is defined on an open disc $U \subset \C$ centred at 
the origin and containing the closed unit disc $\overline \Delta=\{z\in \C:|z|\le 1\}$. 
Set $\theta=z^{k-1}dz$ and write $2\di F_0=h_0\theta$,
where $h_0:U\to Y$ is a $G$-equivariant holomorphic map 
such that $h_0(U\setminus \{0\})\subset \C^* y_0$ for some $y_0\in\nullq_*$.
Let $\pi:\C\to \C/G=\C$ denote the quotient projection $\pi(z)=z^k$.
Then, $\wt U=\pi(U)$ is a disc in $\C$ containing $\overline \Delta$.
By Lemma \ref{lem:correspondence}, $h_0$ determines a holomorphic section 
$\tilde h_0:\wt U \to Z$ of the holomorphic map $\rho:Z=(\C\times Y)/G \to \C$
(see \eqref{eq:Z}) over $\wt U$. Note that $\rho$ is branched only over the origin.
We attach to $\overline \Delta$ the segment $\wt C=[1,2^k]\subset \R\subset\C$.
Note that $\wt S =\overline\Delta\cup\wt C$ is an admissible set.
The preimage $C=\pi^{-1}(\wt C)=\bigcup_{i=1}^k C_i$ is a union of
$k$ pairwise disjoint arcs $C_i$ attached to $\overline \Delta$. 
One of these arcs, say $C_1$, is the segment $[1,2]\subset\R$ and the other
arcs are obtained by rotating $C_1$ through integer multiples of the angle $2\pi/k$.
We define a smooth $G$-equivariant map
$
	f_0 : S=\overline \Delta\cup C \to Y
$
such that $f_0=h_0$ on $\overline\Delta$, $f_0(C)\subset \nullq_*$,
and $f_0(2)\in \nullq_* \setminus \C y_0$. 
Let $\tilde f_0:\wt S=\overline\Delta\cup\wt C \to Z$ denote
the section of $\rho:Z\to\C$ over $\wt S$ associated to $f_0$ 
by Lemma \ref{lem:correspondence}. 
By \cite[Theorem 6.14.6]{Forstneric2017E} there is 
a holomorphic section $\tilde f:\C \to Z$ 
which approximates $\tilde f_0$ on $\wt S$, it 
agrees with $\tilde f_0$ to a given finite order in the origin $0\in\C$, 
it agrees with $\tilde f_0$ in the point $2^k$, and it satisfies
$\tilde f(\C^*) \subset (\C^* \times \nullq_*)/G$.
The section $\tilde f$ defines a $G$-equivariant holomorphic map $f:\C \to Y$ satisfying 
$f(\C^*)\subset \nullq_*$ which agrees with $h_0$ to the given 
order $r$ in the origin and satisfies $f(2)=f_0(2)\in \nullq\setminus \C y_0$. 
Hence, $f$ is nonflat on the arc $C_1$, and therefore on $\C$. By integrating $f\theta$ 
we obtain a nonflat $G$-equivariant conformal minimal immersion $F:\C \to\R^n$
which agrees with $F_0$ to order $r$ in $z=0$ and approximates $F_0$
on $\overline\Delta$.
\end{proof}

%
%
%
%
\section{Proof of Theorem \ref{th:main}}\label{sec:proof}

In this section, we prove the following result.

\begin{theorem}\label{th:mainbis}
Assume that $G$ is a finite group acting effectively on a connected open 
Riemann surface $X$ by holomorphic automorphisms and acting on 
$\R^n$ $(n\ge 3)$ by orthogonal transformations. Let $\pi:X\to X/G$ be the quotient
projection and $X_0\subset X$ be the set \eqref{eq:X0} of its branch points. 
Assume that $S\subset X$ is a $G$-invariant admissible subset 
(see Definition \ref{def:admissible}) which is Runge in $X$,  
$A\subset X\setminus bS$ is a closed $G$-invariant discrete set containing $X_0$, 
$V\subset X\setminus S$ is an open $G$-invariant neighbourhood of $A\setminus S$, 
and $F_0:S\cup V\to \R^n$ is such that $F_0|_S\in \GCMI^r_G(S)$ $(r\ge 1)$ is a nonflat
$G$-equivariant generalized conformal minimal immersion, 
and $F_0|_V$ is a $G$-equivariant conformal minimal immersion. 
Then, $F_0$ can be approximated as closely as desired in the $\Cscr^r$ topology on 
$S$ by $G$-equivariant conformal minimal immersions $F:X\to\R^n$ which agree with $F_0$
to a given finite order $k(a)\in\N$ in every point $a\in A$.
\end{theorem}

By Remarks \ref{rem:disconnectedX} and \ref{rem:disconnectedX2}, 
Theorem \ref{th:mainbis} also holds if $X$ is not necessarily
connected and the stabiliser $G_{X'}\subset G$ of each connected component 
$X'$ of $X$ acts effectively on $X'$.

We have seen in Section \ref{sec:towards} that the hypotheses of Theorem \ref{th:main} 
imply the existence of a $G$-equivariant conformal minimal immersion $F_0:V\to\R^n$
from a $G$-invariant neighbourhood $V\subset X$ of the closed discrete 
set $X_0$ \eqref{eq:X0},
so Theorem \ref{th:mainbis} implies Theorem \ref{th:main}. To obtain
nondegenerate conformal minimal immersions $X\to\R^n$ (i.e., not lying
in any affine hyperplane), it suffices to suitably choose their values 
on finitely many $G$-orbits in $X$ and apply the interpolation statement 
on the set $A\subset X$ in Theorem \ref{th:mainbis}.

\begin{proof}[Proof of Theorem \ref{th:mainbis}] 
By Lemmas \ref{lem:main} and \ref{lem:nonflat} we may assume, after 
shrinking $V$ around the closed discrete set $A\setminus S$ if necessary, that $F_0$ is a 
nonflat $G$-equivariant conformal minimal immersion $F_0:U\cup V\to\R^n$, 
where $U\subset X$ is a $G$-invariant open neighbourhood of $S$ such that 
$A\cap S=A\cap U$ and $\overline U\cap \overline V=\varnothing$. Set 
\[
	\wt X=X/G, \quad \wt X_0=\pi(X_0), \quad \wt A=\pi(A),\quad 
	\wt U=\pi(U),\quad \wt V=\pi(V).
\]
Let $\theta$ be a $G$-invariant holomorphic 1-form \eqref{eq:theta0} with 
$\{\theta=0\}=X_0$, and let $Y$ be the manifold \eqref{eq:Y}.
By Proposition \ref{prop:properties} the map
$f_0=2\di F_0/\theta: U\cup V \to Y$ is holomorphic and $G$-equivariant. 
Let $\tilde f_0$ be the associated section of the map 
$\rho:Z=(X\times Y)/G \to X/G = \wt X$ \eqref{eq:Z} over $\wt U\cup \wt V$
(see Lemma \ref{lem:correspondence}). 
Since $S$ is a $G$-invariant admissible Runge set in $X$, the image
$\wt S=\pi(S)\subset\wt X$ is an admissible set which is Runge in $\wt X$ 
by Lemma \ref{lem:Runge}.
Hence there is a smooth strongly subharmonic Morse exhaustion function 
$\psi:\wt X\to\R_+$ and an increasing sequence $0<c_0<c_1<\cdots$ with 
$\lim_{i\to\infty}c_i=+\infty$ such that, setting $D_i=\{\psi \le c_i\}$, 
we have that $\wt S\subset \mathring D_0\subset D_0\subset \wt U$
and the following conditions hold for every $i\in\Z_+$. 
\begin{itemize} 
\item The number $c_i$ is a regular value of $\psi$.
\item $\{\psi=c_i\}\cap \wt A=\varnothing$.
\item The domain $\Gamma_{i+1}=\mathring D_{i+1}\setminus D_{i}$
contains at most one critical point of $\psi$ or at most one point of $\wt A$, but not both.
\end{itemize} 
For every $i\in\Z_+$ we set 
\begin{equation}\label{eq:Bi}
	B_i=\pi^{-1}(D_i)=\{\psi\circ \pi \le c_i\}\subset X.
\end{equation}
Note that the smoothly bounded compact sets $B_i$ are $G$-invariant and
they form a normal exhaustion of $X$.

To prove the theorem, we shall inductively construct a sequence 
$(F_i,f_i\theta)\in \GCMI_G^r(B_i,\R^n)$ of nonflat, $G$-equivariant   
generalized conformal minimal immersions satisfying 
the following two conditions for every $i\in \Z_+$:
\begin{enumerate}[\rm (a)]
\item $F_{i+1}$ approximates $F_{i}$ in the $\Cscr^r$ topology as closely as desired
on $B_{i}$, and
\item $F_{i+1}$ agrees with $F_0$ to the given order $k(a)$ 
in every point $a\in A\cap B_{i+1}$.
\end{enumerate}
Assuming that the approximation conditions are appropriately chosen, 
the sequence $F_i$ converges to a $G$-equivariant conformal minimal immersion
$F=\lim_{i\to\infty}F_i:X\to \R^n$ satisfying Theorem \ref{th:mainbis}.
We refer to \cite[proof of Theorem 3.6.1]{AlarconForstnericLopez2021} for the details.

The initial step is provided by the restriction of  
$(F_0,f_0\theta)$ to $B_0$. Assuming inductively that we have found 
$(F_i,f_i\theta)\in \GCMI_G^r(B_i,\R^n)$ for some $i\in\Z_+$, we shall explain how to find 
the next map $(F_{i+1},f_{i+1}\theta)\in \GCMI_G^r(B_{i+1},\R^n)$ 
with the desired properties. Recall that 
$
	\Gamma_{i+1}=\mathring D_{i+1}\setminus D_i
$
for $i\in\Z_+$, so $D_{i+1}=D_i\cup\overline \Gamma_{i+1}$. We consider cases.

\smallskip\noindent{\em Case 1: $\Gamma_{i+1}$
does not contain any critical point of $\psi$ or a point of $\wt A$.}
In this case, $D_i$ is a strong deformation retract of $D_{i+1}$, and 
a generalized conformal minimal immersion 
$(F_{i+1},f_{i+1}\theta) \in \GCMI_G^r(B_{i+1},\R^n)$  
with the desired properties is furnished by Lemma \ref{lem:main}.

\smallskip
\noindent{\em Case 2: $\Gamma_{i+1}$ contains a critical point $\tilde x$ of $\psi$.}
We can attach to $D_i$ a smooth embedded arc 
$\wt E\subset \mathring D_{i+1}\setminus \mathring D_i$, intersecting $D_i$ only at 
its endpoints $\tilde p$ and $\tilde q$, such that $\wt S_i = D_i\cup \wt E$ 
is an admissible set  (see Definition \ref{def:admissible})
which is a strong deformation retract of $D_{i+1}$. 
(One can distinguish several topologically different subcases as in  
\cite[proof of Theorem 3.6.1]{AlarconForstnericLopez2021}, but this will not
affect our discussion.) Since $\overline \Gamma_{i+1}\cap \wt X_0=\varnothing$, 
the map $\pi$ is unbranched over $\overline \Gamma_{i+1}$ and hence 
the preimage $E=\pi^{-1}(\wt E)=\bigcup_{j=1}^m E_j \subset X$ 
is a disjoint union of $m=|G|$ smooth arcs. Let $p_j,q_j\in bB_i$ 
denote the endpoints of $E_j$ with $\pi(p_j)=\tilde p_j$
and $\pi(q_j)=\tilde q_j$ for $j=1,\ldots,m$. 
By \cite[Lemma 3.5.4]{AlarconForstnericLopez2021}
we can extend the given map $f_i:B_i\to Y$ from the induction step 
to a nonflat $G$-equivariant map $f'_i: S_i=B_i\cup E \to Y$ of class $\Ascr^{r-1}$
such that $f'_i(E)\subset \nullq_*$ and 
\begin{equation}\label{eq:period2}
	\Re \int_{E_j} f'_i\theta = F_i(q_j)-F_i(p_j)\ \ \text{holds for $j=1,\ldots,m$}, 
\end{equation}
where the arc $E_j$ is oriented from $p_j$ to $q_j$. 
(It suffices to ensure the condition \eqref{eq:period2}
on the arc $E_1$, as it then extends by $G$-equivariance to the remaining 
arcs $E_2,\ldots,E_m$. See \cite[proof of Theorem 3.6.1, p.\ 158]{AlarconForstnericLopez2021}
for the details of this argument.) By Remark \ref{rem:GCMI},
the map $f'_i$ determines an extension $F'_i$ of $F_i$ 
to the $G$-invariant admissible set $S_i=B_i\cup E=\pi^{-1}(\wt S_i)$
such that $(F'_i,f'_i\theta) \in \GCMI_G(S_i)$. 
By Lemma \ref{lem:main} we can approximate $(F'_i,f'_i\theta)$
in the $\Cscr^{r-1}(S_i)$ topology by $(F_{i+1},f_{i+1}\theta)\in \GCMI_G^r(B_{i+1},\R^n)$ 
having the desired properties.

\smallskip
\noindent{\em Case 3: $\Gamma_{i+1}$ contains a point $\tilde a \in \wt A$.} 
Let $\wt \Delta \subset \Gamma_{i+1}\cap \wt V$ 
be a small closed disc around $\tilde a$. The initial map $F_0$ is then a conformal 
$G$-equivariant minimal immersion on $\Delta=\pi^{-1}(\wt \Delta)\subset V$. 
Note that $B_i\cap \Delta=\varnothing$. We extend $(F_i,f_i\theta)$ to $B_i\cup \Delta$ 
by setting $F_i=F_0$ and $f_i=f_0$ on $\Delta$.
Choose a point $\tilde q\in b\wt \Delta$ and attach to $D_i$ a smooth embedded arc 
$\wt E \subset \mathring D_{i+1}\setminus (\mathring D_i \cup\mathring {\wt \Delta})$, 
having an endpoint $\tilde p\in bD_i$ and the other endpoint $\tilde q\in b\wt \Delta$, 
such that $\wt S_i:= D_i\cup \wt \Delta \cup \wt E$ is an admissible set in $\wt X$.
Note that $\wt S_i$ is a strong deformation retract of $D_{i+1}$. 
Set $S_i=\pi^{-1}(\wt S_i)\subset X$.
We now proceed as in Case 2, first extending $(F_i,f_i\theta)$ from $B_i\cup \Delta$ 
to $(F'_i,f'_i\theta)\in \GCMI^r_G(S_i)$ such that conditions \eqref{eq:period2} hold
and then applying Lemma \ref{lem:main} to obtain 
$(F_{i+1},f_{i+1}\theta)\in \GCMI_G^r(B_{i+1},\R^n)$ with the desired properties.

This completes the induction step, and hence of the proof of Theorem \ref{th:mainbis}. 
\end{proof}

%
%
\begin{remark}[Controlling the flux]\label{rem:flux}
In the proof of Theorem \ref{th:mainbis}, one can also control
the flux of $G$-equivariant conformal minimal immersions, provided that the
flux homomorphism is $G$-equivariant. Explicitly, in the proof of Cases 2 and 3 above,
we can use \cite[Lemma 3.5.4]{AlarconForstnericLopez2021}
to extend the map $f_i:B_i\to Y$ 
to a nonflat $G$-equivariant map $f'_i: S_i=B_i\cup E \to Y$ of class $\Ascr^{r-1}$
such that $f'_i(E)\subset \nullq_*$, condition \eqref{eq:period2} holds, and in 
addition the imaginary parts $\Im \int_{E_j} f'_i\theta$ for $j=1,\ldots,m$ 
assume any given set of $G$-equivariant values in $\R^n$.
In particular, we can obtain the following analogue of Theorem \ref{th:main} 
for $G$-equivariant immersed holomorphic null curves $H=(H_1,\ldots,H_n): X \to \C^n$,
i.e., $G$-equivariant holomorphic immersions satisfying $\sum_{i=1}^n (dH_i)^2=0$.
(See \cite[Theorem 3.6.1]{AlarconForstnericLopez2021}
for the basic case when $G$ is the trivial group.) 

%
%
\begin{theorem}[$G$-equivariant null holomorphic curves] \label{th:null}
Let $G$ be a finite group acting effectively on a connected open 
Riemann surface $X$ by holomorphic automorphisms, and acting on 
$\C^n$ $(n\ge 3)$ by complex orthogonal transformations in $O(n,\C)$. 
Assume that for every nontrivial stabiliser $G_x$ $(x\in X)$ there is a $G_x$-invariant 
null complex line $L_x\subset \nullq\subset\C^n$ on which a generator of 
$G_x$ acts by $z\mapsto \E^{\imath \phi}z$ with $\phi=2\pi/|G_x|$. 
Then there exists a nondegenerate $G$-equivariant holomorphic null immersion 
$H:X\to \C^n$.
\end{theorem}
\end{remark}

%
%
\section{Symmetric minimal surfaces with ends of finite total curvature}\label{sec:FTC}

Let $X$ be an open Riemann surface. An immersed minimal surface $F:X\to \R^n$ 
has nonpositive Gaussian curvature function $\Kcal :X\to (-\infty,0]$. 
The surface is said to have {\em finite total curvature} if 
\[
	\int_X \Kcal \, d\sigma>-\infty.
\]
Here, $d\sigma$ denotes the surface area in the induced Riemannian metric 
$F^*ds^2$ on $X$, where $ds^2$ is the Euclidean metric on $\R^n$. 
If in addition $F$ is {\em complete}, meaning that the pullback 
$F^*ds^2$ of the Euclidean metric is a complete metric on $X$,
then by Huber \cite{Huber1957CMH} the surface $X$ is biholomorphic to 
$\Sigma \setminus P$ where $\Sigma$ is a compact Riemann surface and $P$ 
is a nonempty finite subset of $\Sigma$. Furthermore, by
the Chern--Osserman theorem \cite{ChernOsserman1967JAM}, $\di F$ is a meromorphic 
$\nullq$-valued $1$-form on $X$ with an effective pole in  
every point of $P$, and the map $F:X\to \R^n$ is proper.
Many classical minimal surfaces are of this kind.
Complete minimal surfaces of finite total curvature have been of major interest
in the theory since the early works of Osserman in the 1960's; see 
\cite{ChernOsserman1967JAM,JorgeMeeks1983T,Osserman1986,BarbosaColares1986LNM,Yang1994MA,AlarconForstnericLopez2021} for background on this subject. 

To motivate our next result (see Theorem \ref{th:FTC}), we first 
recall the following Runge-type theorem 
for complete minimal surfaces with finite total curvature, 
due to Alarc\'on and L\'opez \cite[Theorem 1.2]{AlarconLopez2022APDE}.
A proof can also be found in \cite[Theorem 4.5.1]{AlarconForstnericLopez2021}, 
and a different proof has been given recently by Alarc\'on and L\'arusson 
\cite[Corollary 5.3]{AlarconLarusson2023X}.

%
%
\begin{theorem}
\label{th:AL-FTC} 
Let $\Sigma$ be a compact Riemann surface, $P$ be a nonempty 
finite subset of $\Sigma$, $K$ be a smoothly bounded  
compact Runge domain in the open Riemann surface $X=\Sigma\setminus P$, 
and let $A$ and $\Lambda$ be disjoint finite subsets of $\mathring K$.
If $F_0:K\setminus A \to \R^n$ $(n\geq 3)$ is a complete conformal minimal immersion with finite total curvature, then for any $\epsilon>0$ and integer $k\ge 0$ there is a conformal minimal immersion $F: X\setminus A \to \R^n$ satisfying the following conditions.  
\begin{enumerate}[\rm (i)]
\item $F-F_0$ extends to a harmonic map on $\mathring K$ and satisfies 
$|F-F_0|<\epsilon$ on $K$. 
\item $F-F_0$ vanishes at least to order $k$ in every point of $A\cup \Lambda$.
\item $F$ is complete and has finite total curvature.
\end{enumerate}
\end{theorem}

The conditions in Theorem \ref{th:AL-FTC} imply that $\di F_0$ has an effective pole 
in every point of $A$, and $F$ has a complete end with finite total curvature 
in every point of $A\cup P$. Clearly, the result is equivalent to 
the special case when $P=\{p\}$ is a singleton. We emphasize that 
one can not prescribe the pole of $\di F$ at $p$ in an arbitrary way. 

One may wonder whether an analogue of Theorem \ref{th:AL-FTC} holds 
in the $G$-equivariant case. 
Explicitly, assume that $\Sigma$ is a compact Riemann surface, 
$G\subset \Aut(\Sigma)$ is a finite group of automorphisms, and
$P$ is a finite $G$-invariant subset of $\Sigma$. Then, $G$ also acts on the 
open Riemann surface $X=\Sigma\setminus P$,  
and the set $X_0=\{x\in X: G_x\ne \{1\}\}$ \eqref{eq:X0} is finite. 
Let $K\subset X$ be a smoothly bounded Runge compact set 
containing $X_0$ in the interior, and let $A$ and $\Lambda$ be disjoint 
finite $G$-invariant subsets of $\mathring K$. 
Assume that $G$ also acts on $\R^n$ by orthogonal maps. 
Under these assumptions, we pose the following problem.

\begin{problem}\label{prob:FTC} 
Given a $G$-equivariant complete conformal minimal immersion 
$F_0:K\setminus A \to \R^n$ with finite total curvature, 
is there a $G$-equivariant complete conformal minimal immersion 
$F: X\setminus A \to \R^n$ with finite total curvature satisfying Theorem \ref{th:AL-FTC}?
\end{problem}

We do not know the answer to this question, except 
for the results of Xu \cite{Xu1995} pertaining to planar domains
(i.e., minimal surfaces of genus zero, see Example \ref{ex:sphere}). 
However, the methods used in the proof of Theorem \ref{th:mainbis} give
the following result in this direction.

%
%
\begin{theorem}
\label{th:FTC}
Assume that $X_0\subset X$ and $G$ are as in Theorem \ref{th:mainbis}, satisfying 
the condition on stabilisers of points $x\in X_0$.
Let $K$ be a smoothly bounded, $G$-invariant, Runge compact domain in $X$
such that $X_0\cap bK=\varnothing$, and let $A\subset \mathring K\setminus X_0$ 
and $\Lambda\subset \mathring K$ be disjoint $G$-invariant finite sets.
Given a complete $G$-equivariant conformal minimal immersion
$F_0:K\setminus A \to \R^n$ $(n\geq 3)$ of finite total curvature, 
there exists for any $\epsilon>0$ and integer $k \ge 0$ a 
nondegenerate $G$-equivariant conformal minimal immersion 
$F: X\setminus A \to \R^n$ satisfying the following conditions.  
\begin{enumerate}[\rm (a)]
\item $F-F_0$ extends to a harmonic map on $\mathring K$ and 
$|F-F_0|<\epsilon$ on $K$. 
\item $F-F_0$ vanishes at least to order $k$ in every point of $A \cup \Lambda$.
\item For any compact, smoothly bounded, $G$-invariant set $L$ with 
$K\subset L\subset X$, the $G$-equivariant conformal minimal immersion 
$F:L\setminus A\to\R^n$ is complete and has finite total curvature. 
\end{enumerate}
\end{theorem}

Note that condition (c) follows from (a) and (b). 
The resulting minimal surface $F(X\setminus A)\subset \R^n$ is $G$-invariant
and has a complete end with finite total curvature in every point of $A$.  

\begin{proof}
Let $\theta$ be a $G$-invariant holomorphic 1-form \eqref{eq:theta0} on $X$ with 
$\{\theta=0\}=X_0$, and let $Y$ be the manifold \eqref{eq:Y}.
By Proposition \ref{prop:properties} the map
$f_0=2\di F_0/\theta: K \to Y$ is holomorphic and $G$-equivariant,
and it maps the finite set $A\cup (K\cap X_0)$ to $Y_0$ \eqref{eq:Y0}. 
(That is, the restricted map $f_0:K\setminus (A\cup X_0)\to \nullq_*$ is holomorphic and
has an effective pole in every point of $A\cup (K\cap X_0)$.)
Let $\tilde f_0$ be the associated holomorphic section of the map 
\[
	\rho:Z=(X\times Y)/G \to X/G 
\]
(see \eqref{eq:Z}) over $K$. 
An inspection of the proof of Theorem \ref{th:mainbis} shows that it
applies without any changes in this situation, and the resulting $G$-equivariant
conformal minimal immersion $F:X\setminus A\to\R^n$ can be chosen 
such that it satisfies the conclusion of the theorem. Further details
in the non-equivariant situation can be found in 
\cite[proof of Theorem 1.3]{AlarconLopez2022APDE}.
\end{proof}

%
%
\begin{remark} (A) Theorem \ref{th:FTC} shows that, in the setting of Problem \ref{prob:FTC},
we can find a $G$-equivariant conformal minimal immersion $F:X=\Sigma\setminus P\to\R^n$
with a complete end of finite total curvature in every point of $P\setminus P_0$,
where $P_0\subset P$ is an orbit of $G$.
 
(B) An inductive application of Theorem \ref{th:FTC} gives a  
Mittag--Leffler type theorem for $G$-equivariant minimal surfaces 
having a complete finite total curvature end in every point of a 
given closed discrete $G$-invariant subset $A$ of $X$
with $A\cap X_0=\varnothing$, provided that the prescription of the map in the ends 
(i.e., in the points of $A$) is $G$-equivariant. The analogous result without
group equivariance is due to 
Alarc\'on and L\'opez \cite[Theorem 1.3]{AlarconLopez2022APDE}.

We do not know whether Theorem \ref{th:FTC} still holds if the sets $A$ 
and $X_0$ are not disjoint, for in this case a problem appears in the 
proof of Lemma \ref{lem:main}.
\end{remark}

The following immediate corollary to Theorem \ref{th:FTC}
shows that any type of complete ends  
with finite total curvature can be realised by a $G$-equivariant conformal 
minimal immersion.

\begin{corollary}\label{cor:FTC}
Assume that $X_0\subset X$ and $G$ are as in Theorem \ref{th:main}. 
Let $x_1,\ldots, x_m\in X\setminus X_0$ be points in distinct $G$-orbits,
and for each $i=1,\ldots,m$ let $F_i:U_i\setminus \{x_i\}\to\R^n$
be a conformal minimal immersion of finite total curvature defined on a 
punctured neighbourhood of $x_i$. Set $A=\bigcup_{i=1}^m G x_i$. 
Then there exists a nondegenerate $G$-equivariant conformal minimal immersion 
$F:X\setminus A\to\R^n$ which agrees with $F_i$ to any given finite order 
in $x_i$ for $i=1,\ldots,m$.
\end{corollary}

%
%
%
%
\section{Minimal surfaces with infinite discrete groups of symmetries}\label{sec:infinite}

Let $X$ be an open Riemann surface, and let $G$ be an infinite discrete group acting 
on $X$ by holomorphic automorphisms such that the action is 
{\em properly discontinuous}, meaning that every pair of points $p,q\in X$
admits open neighbourhoods $p\in U\subset X$, $q\in V\subset X$ such that
the set $\{g\in G:gU \cap V\ne \varnothing\}$ is finite. (Equivalently, 
for every compact subset $K\subset X$
the set $\{g\in G: gK \cap K\ne \varnothing\}$ is finite.) 
Then, $G$ is countable, every orbit $Gx$ $(x\in X)$ is an infinite closed 
discrete subset of $X$,
the set $X_0=\{x\in X: G_x\ne \{1\}\}$ \eqref{eq:X0} is closed and discrete in $X$, 
every nontrivial stabiliser $G_x$ $(x\in X_0)$ is a finite cyclic group, 
the quotient space $X/G$ has the structure of a Riemann surface, and  
the quotient projection $\pi:X\to X/G$ is holomorphic. 
(See Miranda \cite[p.\ 83]{Miranda1995} and tom Dieck 
\cite[Chapter 1, Sect.\ 3]{tomDieck1987TG} for these facts.) 

Assume that $G$ also acts on $\R^n$ by {\em rigid transformations},
i.e., compositions of orthogonal maps, translations, and dilations.
We have the following generalization of Theorem \ref{th:main}.

%
%
\begin{theorem}\label{th:infinite}
Let $G$ be a countable discrete group acting properly discontinuously on an open  
Riemann surface $X$ by holomorphic automorphisms such 
that the Riemann surface $X/G$ is open, 
and acting on $\R^n$ $(n\ge 3)$ by rigid transformations. 
If for every point $x\in X$ with nontrivial stabiliser
$G_x$ there is a $G_x$-invariant affine $2$-plane $\Lambda\subset \R^n$ on which 
$G_x$ acts effectively by rotations, then there exists a nondegenerate 
$G$-equivariant conformal minimal immersion $X\to\R^n$.
\end{theorem}

\begin{proof}[Sketch of proof]
It suffices to inspect the proof of Theorem \ref{th:mainbis}, taking into 
account the information given above. We point out the relevant 
facts and the necessary modifications.

We obtain a $G$-invariant holomorphic $1$-form on $X$ as before by taking
$\theta=d(\tilde h\circ\pi)$ \eqref{eq:theta0}, 
where $\tilde h:X/G\to\C$ is a holomorphic immersion.
Note that the property \eqref{eq:theta} of $\theta$ still holds. 

By the hypothesis, every $g\in G$ acts on $\R^n$ by a map of the form
\begin{equation}\label{eq:gonRn} 	
	g\bx=rO\bx+b,\quad \bx\in\R^n
\end{equation}
for some $r>0$, $O\in O(n,\R)$, and $b\in \R^n$. 
Its differential $dg=rO$ also acts on $\C^n$, $\CP^n$, 
the null quadric $\nullq$, and the manifold $Y=\nullq_*\cup Y_0$ \eqref{eq:Y}. 
The formulas \eqref{eq:Fg0}--\eqref{eq:fg} 
must be adjusted by replacing $g$ acting on the derivative 
$\di F$ and on the map $f=2\di F/\theta:X\to Y$ by
its differential $dg=rO$, so the correct analogue of the equations 
\eqref{eq:Fg}--\eqref{eq:fg} is
\begin{equation}\label{eq:fdg}
	\di F_{gx}\circ dg_x = dg\circ \di F_x,\qquad f(gx) = dg\,\cdotp f(x)=rO\,\cdotp f(x).
\end{equation}
Everything said in the sequel assumes this notion of $G$-equivariance of the map $f$.
Remark \ref{rem:necessary} applies verbatim and shows 
that the condition on stabilisers in Theorem \ref{th:infinite} is necessary. 
(Note that for $g\in G_x$ the differential $dg$ of its action on $\R^n$ \eqref{eq:gonRn} 
cannot contain a dilation by a factor $0<r\ne 1$ since the group $G_x$ is finite.)
Proposition \ref{prop:properties} still holds, except that 
the second part of condition (c) (that $F(x_0)$ is orthogonal to $\Lambda$)
need not hold since $g_0$ can be conjugate to a rotation by a translation. 
Lemma \ref{lem:Gequivariance} and Theorem \ref{th:GWR} remain unchanged.
The action \eqref{eq:gxy} of the group $G$ on $X\times Y$ is now replaced by
\[
	g(x,y)=(gx,dg\,\cdotp y),\quad x\in X,\ y\in Y,\ g\in G.
\]
This action is properly discontinuous 
(since it is such on $X$), so the quotient $Z=(X\times Y)/G$ is a 
reduced complex space and the projection 
$\rho:Z=(X\times Y)/G \to X/G=\wt X$ \eqref{eq:Z} is a holomorphic map
which is ramified (only) over $\pi(X_0)=\wt X_0$, and 
$\rho:Z\setminus \rho^{-1}(\wt X_0) \to \wt X\setminus \wt X_0$ is a
holomorphic $G$-bundle with fibre $Y$ as before.
The correspondence between 
$G$-equivariant maps $f:X\to Y$ (in the sense of \eqref{eq:fdg})
and sections $\tilde f:X\to Z$ of $\rho$, given by Lemma \ref{lem:correspondence}, 
remains valid. Lemma \ref{lem:onV} holds without changes.

In Lemma \ref{lem:main}, the $G$-invariant admissible set $S$ is no
longer compact, but its projection $\wt S=\pi(S)\subset \wt X$ is compact,
which is all that matters in the proof. In fact, we see from 
\eqref{eq:period-i} that the period domination property of the spray $h_i$
on the arcs $C_{i,j}$ in $C_i=\pi^{-1}(\wt C_i)=\bigcup_{j} C_{i,j} \subset S$ 
must only be arranged on one the arcs, $C_{i,1}$, as it extends 
by $G$-equivariance to a spray $C_i\times B\to Y$. 
(Here, $j\in\N$ is the index corresponding to the elements of $G$.) 
Lemma \ref{lem:nonflat} is of local nature and holds without changes.

The proof of Theorem \ref{th:mainbis} (see Section \ref{sec:proof}) is carried out with 
respect to a normal exhaustion of the open Riemann surface $\wt X=X/G$ by an increasing
sequence of compact Runge subsets $D_i$. In light of 
the previous discussion, it is seen by inspection that all steps hold, 
so we obtain Theorem \ref{th:infinite}.
Note that locally uniform convergence of the resulting sequence of 
holomorphic sections $\tilde f_i:D_i\to Z$ implies uniform convergence  
of the corresponding sequence $F_i:B_i=\pi^{-1}(D_i)\to\R^n$ of $G$-equivariant 
conformal minimal immersions on compacts in $X$.
However, when the group $G$ is infinite, 
there are no nonempty compact $G$-invariant sets in $X$.
Hence, an exact analogue of Theorem \ref{th:mainbis} 
(with uniform approximation on a $G$-invariant set) 
is not possible if dilations are involved in the action of $G$ on $\R^n$.
\end{proof}

Theorem \ref{th:infinite} implies the following analogue of Corollary \ref{cor:free}.

\begin{corollary}\label{cor:infinitefree}
If $G$ is an infinite discrete group acting freely and properly discontinuously on an open 
Riemann surface $X$ by holomorphic automorphisms such that the Riemann
surface $X/G$ is open, then for every action of 
$G$ on $\R^n$ $(n\ge 3)$ by rigid transformations there exists a 
nondegenerate $G$-equivariant 
conformal minimal immersion $X\to \R^n$, which can be chosen to be the
real part of a $G$-equivariant null holomorphic immersion $X\to\C^n$.
\end{corollary}

The analogue of Corollary \ref{cor:special} also holds for an infinite discrete
group of rigid transformations on $\R^n$ whose induced action on an oriented 
embedded surface $X\subset \R^n$ is properly discontinuous and 
the orbit space $X/G$ is noncompact.
The statement of Theorem \ref{th:FTC} can also be adjusted to infinite
discrete groups acting on $X$ as in Theorem \ref{th:infinite}. 

As mentioned in the introduction, every group of automorphisms
of a Riemann surface of genus $\ge 2$ is finite. Hence, Theorem \ref{th:infinite} 
and its corollaries 
are relevant only when 
$X$ is an open domain in $\C$ or in a complex torus. Natural examples 
of such actions are by groups of translations on $X=\C$ and on $\R^n$, 
and there are several classical examples of translation-invariant 
minimal surfaces (for example, the helicoid, Scherk's surfaces, Schwarz's surfaces,
to name a few of the best known ones).
Also, every Riemann's minimal surface in $\R^3$ is translation invariant and 
parameterised by a domain in $\C$ on which the group $\Z$ acts properly 
discontinuously; see the survey by Meeks and Perez \cite{MeeksPerez2016}
and \cite[Subsec.\ 2.8.5]{AlarconForstnericLopez2021}. 
Another natural case is when $X$ is the unit disc $\D=\{|z|<1\}$ 
(or, equivalently, the upper halfplane $\H$), since 
its group of holomorphic automorphisms contains many infinite discrete subgroups
acting properly discontinuously (and even freely).

%
%
%
%
\smallskip
\noindent {\bf Acknowledgements.} 
Research is supported by the European Union (ERC Advanced grant HPDR, 101053085) 
and by grants P1-0291, J1-3005, and N1-0237 from ARIS, Republic of Slovenia. 
I wish to thank Antonio Alarc\'on for helpful comments, 
Urban Jezernik for an idea used in the proof of Corollary \ref{cor:everyG}, 
Frank Kutzschebauch for helpful discussions concerning the Oka principle for 
$G$-equivariant holomorphic maps, Francisco J.\ L\'opez for communicating the 
question which is partially answered in the paper, and Tja\v sa  Vrhovnik for a careful
reading of the draft and pointing out several necessary corrections.
Last but not least, I thank the referee for careful reading 
and inspiring suggestions.




\vspace*{5mm}
\noindent Franc Forstneri\v c

\noindent Faculty of Mathematics and Physics, University of Ljubljana, Jadranska 19, SI--1000 Ljubljana, Slovenia

\noindent 
Institute of Mathematics, Physics and Mechanics, Jadranska 19, SI--1000 Ljubljana, Slovenia

\noindent e-mail: {\tt franc.forstneric@fmf.uni-lj.si}

\end{document}